\documentclass [12pt,a4paper, reqno,intlimits]{article}

\usepackage{amstext,amsmath,amssymb,amsthm,ascmac,cases}
\usepackage{braket}
\usepackage{amscd}
\usepackage{fancyhdr}
\usepackage[final]{pdfpages}
\usepackage{graphicx,color}

\usepackage{tikz}
\usetikzlibrary{positioning,shapes,arrows}
\usetikzlibrary{3d,calc}
\tikzstyle{na} = [baseline=-.5ex]
\usetikzlibrary{intersections,math,patterns}
\usepackage{pgfplots}
\pgfplotsset{compat=1.6}

\pgfplotsset{soldot/.style={color=blue,only marks,mark=*}}
\pgfplotsset{holdot/.style={color=blue,fill=white,only marks,mark=*}}

 \usepackage[colorlinks=true,linkcolor=blue, setpagesize=false, linkbordercolor=blue]{hyperref}
\hypersetup{urlcolor=green, citecolor=green}

\usepackage{amsgen}

\usepackage{mathrsfs}

\usepackage{geometry}
\usepackage{wrapfig}

\usepackage{enumerate}

\usepackage{dashbox}

\DeclareMathOperator*{\esssup}{ess\,sup}

\DeclareMathOperator*{\supp}{supp}

\usepackage{esint}

\def\Xint#1{\mathchoice
{\XXint\displaystyle\textstyle{#1}}%
{\XXint\textstyle\scriptstyle{#1}}%
{\XXint\scriptstyle\scriptscriptstyle{#1}}%
{\XXint\scriptscriptstyle\scriptscriptstyle{#1}}%
\!\int}
\def\XXint#1#2#3{{\setbox0=\hbox{$#1{#2#3}{\int}$ }
\vcenter{\hbox{$#2#3$ }}\kern-.6\wd0}}

\def\dashint{\Xint-}

\makeatletter
  \def\thefootnote{\ifnum\c@footnote>\z@\leavevmode\lower.5ex%
      \hbox{$^{\@arabic\c@footnote)}$}\fi}
  \makeatother

\newtheoremstyle{mystyle}
  {}
  {}
  {\itshape}
  {}
  {\bfseries}
  { }
  { }
  {\thmname{#1}\thmnumber{ #2}\thmnote{ (#3)}}

\theoremstyle{mystyle}

\usepackage[title, toc, titletoc]{appendix}

\allowdisplaybreaks[4]

\newtheorem{thm}{Theorem}[section]
\newtheorem{prop}[thm]{Proposition}
\newtheorem{lem}[thm]{Lemma}
\newtheorem{dfn}[thm]{Definition}

\newtheorem{rmk}[thm]{Remark}

\usepackage{titlesec}

\makeatletter

\@addtoreset{equation}{section}
\makeatother

\renewcommand{\thefigure}{\arabic{section}.\arabic{figure}}

\renewcommand{\to}{\longrightarrow}

\usepackage{authblk}
\usepackage{blindtext}

\begin{document}
\titlepage
\title{\Large \textbf{Global existence for the $p$-Sobolev flow}}

\author[$\star$]{Tuomo Kuusi}
\affil[$\star$]{\small Department of Mathematics and Statistics,
University of Helsinki, PL 68 (Pietari Kalmin katu 5)
00014, Finland
 \quad \textcolor{blue}{\texttt{tuomo.kuusi@helsinki.fi}}}

  \author[$\dagger$]{Masashi Misawa}
 \affil[$\dagger$]{\small Faculty of Advanced Science and Technology, Kumamoto University, Kumamoto 860-8555, Japan
  \quad \textcolor{blue}{\texttt{mmisawa@kumamoto-u.ac.jp}}}

  \author[$\ast$]{Kenta Nakamura}
\affil[$\ast$]{\small Headquarters for Admissions and Education, Kumamoto University, Kumamoto 860-8555, Japan
 \quad \textcolor{blue}{\texttt{kntkuma21@kumamoto-u.ac.jp}}}

 \date{}
 \maketitle

\providecommand{\keywords}[1]{\textbf{Keywords---} #1}
\medskip

\providecommand{\msc}[1]{\textbf{MSC2010---} #1}
\begin{msc}
Primary: 35B45, 35B65, \quad Secondary: 35D30, 35K61.
\end{msc}
 \begin{abstract}
In this paper, we study a doubly nonlinear parabolic equation arising from the gradient flow for $p$-Sobolev type inequality, referred as $p$-Sobolev flow. In the special case $p=2$ our theory includes the classical Yamabe flow on a bounded domain in Euclidean space. Our main aim is to prove the global existence of the $p$-Sobolev flow together with its qualitative properties. 
\end{abstract}

\tableofcontents

 \section{Introduction}

Let $\Omega \subset \mathbb{R}^{n}\,(n \geq 3)$ be a bounded domain with smooth boundary $\partial \Omega $. For any positive $T \leq \infty$, let $\Omega_T : = \Omega \times (0, T)$ be the space-time cylinder.
Throughout the paper we fix  $ p \in [2,n)$ and set $q:=p^\ast - 1$, where $p^\ast := \frac{np}{n-p}$ is the Sobolev conjugate of $p$.
We consider the following doubly nonlinear parabolic system
\begin{equation}\label{pS0}
\begin{cases}
\,\,\partial_t(|u|^{q-1}u)-\Delta_pu=\lambda(t) |u|^{q-1}u\quad &\textrm{in}\quad \Omega_\infty  \\
\,\,\displaystyle \| u(t) \|_{L^{q+1}(\Omega)} =1\quad &\textrm{for all}\,\, t >0 \\[1mm]
\,\,u=0\quad &\textrm{on}\quad \!\!\partial\Omega \times (0,\infty) \\
\,\,u(\cdot, 0)=u_0(\cdot ) \quad &\textrm{in}\quad \Omega.
\end{cases}
\end{equation}
Here the unknown function $u=u(x,t)$ is a real-valued function defined for $(x,t) \in \Omega_\infty$, and the initial data $u_0$ is assumed to be in the Sobolev space $W^{1, p}_0 (\Omega)$, positive, bounded in $\Omega$ and satisfy $ \| u_0 \|_{L^{q+1}(\Omega)}=1$, as usual, $\partial_t:= \partial / \partial t$ and $\partial_\alpha:= \partial / \partial x_\alpha, \alpha = 1, \ldots, n$ are the partial derivatives on time and space, respectively, and $\nabla:= (\partial_1,\ldots ,\partial_n)$ is the gradient on space, and $\Delta_p u:=\mathrm{div}\left(|\nabla u|^{p-2}\nabla u\right)$ is the $p$-Laplacian. The condition imposed in the second line of~\eqref{pS0} is called the \emph{volume constraint} and $\lambda (t)$ is Lagrange multiplier stemming from this volume constraint. Indeed, multiplying~\eqref{pS0} by $u$ and integrating by parts, we find by a formal computation that $\displaystyle \lambda (t)=\int_\Omega |\nabla u(x,t)|^p \,dx$ (See \cite[Proposition 5.2]{Kuusi-Misawa-Nakamura} for its proof). We call the system~\eqref{pS0} as \textit{$p$-Sobolev flow}.

\medskip

Our main result in this paper is the following theorem. 
 
 \begin{thm}\label{mainthm}
 Assume that the initial value $u_{0}$ belongs to the Sobolev space $W_{0}^{1,p}$,  positive, bounded in $\Omega$, and satisfies $\|u_0\|_{L^{q+1}(\Omega)}=1$. Then there exists a global weak solution to the equation~\eqref{pS0}, which is positive and bounded in $\Omega_\infty$ and is, together with its spatial gradient, locally H\"{o}lder continuous in $\Omega_\infty$.
 \end{thm}
In our forthcoming work we will proceed further with the analysis of the $p$-Sobolev flow. Especially, we will classify the limits as time tends to infinity.

\smallskip
 
The doubly nonlinear equations have been considered by Vespri \cite{Vespri1}, Porzio and Vespri \cite{Porzio-Vespri}, and Ivanov \cite{Ivanov1, Ivanov2}.  See also \cite{Gianazza-Vespri, Vespri2,Kinnunen-Kuusi,Kuusi-Siljander-Urbano}.
The regularity proofs for doubly nonlinear equations are based on the intrinsic scaling method, originally introduced by DiBenedetto, and they have to be arranged in some way depending on the particular form of the equation. In \cite{Kuusi-Misawa-Nakamura} we have already treated the very fast diffusive doubly nonlinear equation such as the $p$-Sobolev flow~\eqref{pS0}, and obtained the positivity, boundedness and regularity of weak solutions. In particular, the expansion of positivity for~\eqref{pS0} is shown by the De Giorgi's iteration based on local energy estimates in the intrinsic scaling setting (refer to \cite{DiBenedetto2}).  The solution to~\eqref{pS} remains positive for all finite times by the volume constraint. This is here applied for the global existence of the $p$-Sobolev flow~\eqref{pS0}, as explained later.


\smallskip

In compact manifold setting with $p=2$, our $p$-Sobolev flow~\eqref{pS} is exactly the classical Yamabe flow equation in the Euclidean space. The classical Yamabe flow was originally introduced by Hamilton in his study of the so-called Yamabe problem (\cite{Yamabe, Aubin1, Aubin2}), asking the existence of a conformal metric of constant curvature on $n(\geq 3)$-dimensional closed Riemannian manifolds (\cite{Hamilton89}). Let $(\mathcal{M},g_{0})$ be a $n(\geq 3)$-dimensional smooth, closed Riemannian manifold with scalar curvature $R_{0}=R_{g_{0}}$. The classical Yamabe flow is given by the heat flow equation
\begin{equation}
\label{heatflow}
u_t=(s-R)u=u^{-\frac{4}{n-2}}(c_n \Delta_{g_0}u-R_0u)+su,
\end{equation}
where $u=u(t), t\geq 0$, is a positive smooth function on $\mathcal{M}$ such that $g(t)=u(t)^{\frac{4}{n-2}}g_0$ is a conformal change of a Riemannian metric $g_{0}$, with volume constraint $\displaystyle \mathrm{Vol}(\mathcal{M})=\int_\mathcal{M}\,dvol_g=\int_\mathcal{M} u^{\frac{2n}{n-2}}dvol_{g_0}=1$, having total curvature
$$\displaystyle s:=\int_\mathcal{M}(c_n |\nabla u|_{g_0}^2+R_0u^2)\,dvol_{g_0}=\int_\mathcal{M}R\,dvol_{g},\quad c_n:=\frac{4(n-1)}{n-2}.$$
Note that the condition for volume above naturally corresponds to the volume constraint in~\eqref{pS0}.
Hamilton (\cite{Hamilton89}) proved a convergence of the Yamabe flow as $t \to \infty$ under some geometric conditions. Under the assumption that $(\mathcal{M},g_0)$ is of positive scalar curvature and locally conformal flat, Ye (\cite{Ye}) showed the global existence of the Yamabe flow and its convergence as $t \to \infty$ to a metric of constant scalar curvature. Schwetlick and Struwe (\cite{Schwetlick-Struwe})
established the asymptotic convergence of the Yamabe flow for an initial positive scalar curvature in the case $3 \leq n \leq 5$, under an appropriate condition of Yamabe invariance $Y(\mathcal{M},g_0)$, which is given by infimum of Yamabe energy $E(u)=\int_\mathcal{M}(c_n |\nabla u|_{g_0}^2+R_0u^2)\,dvol_{g_0}$ among all positive smooth function $u$ on $\mathcal{M}$ with $\mathrm{Vol}(\mathcal{M})=1$. In Euclidean case, since $R_{g_{0}}=0$ their curvature assumptions are not verified. In above outstanding results concerning the Yamabe flow, the equation is equivalently transformed to the scalar curvature equation, and this is crucial for obtaining many properties for the Yamabe flow. In contrast to their methods, we are forced to take a direct approach dictated by the structure of the $p$-Laplacian leading to the degenerate or singular parabolic equation of the $p$-Sobolev flow. Let us remark that our results cover those of the classical Yamabe flow in the Euclidian setting.

\smallskip

Our global existence result for the $p$-Sobolev flow~\eqref{pS0} is established by applying a nonlinear intrinsic scaling transformation to the following prototype doubly nonlinear parabolic equation
\begin{equation}\label{bpS}
\begin{cases}
\,\,\partial_s(|v|^{q-1}v)-\Delta_pv=0\quad &\textrm{in}\quad \Omega_S  \\ 
\,\,v=0\quad &\textrm{on}\quad \!\!\partial\Omega \times (0,S) \\
\,\,v(\cdot, 0)=v_0(\cdot ) \quad &\textrm{in}\quad \Omega.
\end{cases}
\end{equation}
Here $0<S\leq \infty$, the unknown function $v=v(x,s)$ is real-valued function defined for $(x,s) \in \Omega_{S}$, and the given function $v_{0}$ is in the Sobolev space $W^{1, p}_0 (\Omega)$, nonnegative and bounded in $\Omega$. In \cite{Nakamura-Misawa} the existence result is obtained from the backward difference quotient on time and 
%
Galerkin's  procedure. For the global existence we crucially use the expansion of positivity for the $p$-Sobolev flow, as stated before. Based on the positivity estimates in \cite{Kuusi-Misawa-Nakamura}, here we present the refinement for the expansion of positivity for~\eqref{pS0} with its precise proof (See Appendix~\ref{Refinement of Expansion of Positivity}). Combining the nonlinear intrinsic scaling applied for~\eqref{bpS}, and the refined expansion of positivity, we establish the global existence of a regular weak solution to the $p$-Sobolev flow. 
\bigskip

The structure of this paper is as follows. In Section~\ref{Sec. Preliminaries}, we prepare some notation and give the definition of a weak solution of~\eqref{pS}. In Section~\ref{Sec. Prototype doubly nonlinear equation}, we recall the global existence and regularity estimates for~\eqref{bpS} obtained in \cite{Nakamura-Misawa, Kuusi-Misawa-Nakamura}. Starting from positive initial data, the solution of~\eqref{bpS} is positive up to a finite time, that is, the positivity expands and, furthermore, the solution vanishes at a finite time, that is verified by the comparison principle.  In Section~\ref{Sec. Nonlinear intrinsic scaling transformation}, we present the nonlinear intrinsic scaling transformation from~\eqref{bpS} to~\eqref{pS}, which is justified via mollifier argument in Appendix~\ref{Sec. rigorous argument}. In Section~\ref{Sec. Proof of mainthm} we give the proof of Theorem~\ref{mainthm}.  Here the expansion of positivity by the volume constraint for the $p$-Sobolev flow~\eqref{pS0} is crucially applied for extending the life span of the solution and yielding the global existence. In Appendix~\ref{Refinement of Expansion of Positivity}, we present the refined expansion of positivity for the $p$-Sobolev flow type equation by use of a stretching time transformation with its precise proof, and also prove the key propositions used in the proof of Theorem~\ref{mainthm}. In Appendix~\ref{Sec. Convergence result} we give the elementary convergence result with its proof, which is used in the next appendix. In Appendix~\ref{Sec. rigorous argument}, we demonstrate that the nonlinear intrinsic scaling rigorously works.

\subsection*{Acknowledgments}
The nonlinear intrinsic scaling transformation in Section~\ref{Sec. Nonlinear intrinsic scaling transformation} was kindly suggested to us by Professor J. L. Vazquez (\cite{Vazquez3}) in 2013. T. Kuusi is supported by the Academy of Finland (grant 323099) and the European Research Council (ERC) under the European Union's Horizon 2020 research and innovation programme (grant agreement No 818437). \\[2mm]
The work by M. Misawa was partially supported by the Grant-in-Aid
for Scientific Research (C) Grant number No.18K03375 at Japan Society for the Promotion of Science.\\[2mm]
K. Nakamura is supported by Foundation of Research Fellows, The Mathematical Society of Japan. 

 \section{Preliminaries}\label{Sec. Preliminaries}
 \noindent

We prepare some notation and fundamental tools, which are used throughout this paper. 

\bigskip
\noindent

Let $\Omega \subset \mathbb{R}^n\,\,(n \geq 3)$ be a bounded domain with smooth boundary $\partial \Omega$. 
Let us define the parabolic boundary of the space-time cylinder $\Omega_T:=\Omega \times (0,T)$ by
\[ 
\partial_p \Omega_T:= \partial \Omega \times [0, T)\cup \Omega \times \{t = 0\}. 
 \]
 Let $B_R(x_0):=\{x \in \mathbb{R}^n\,:\,|x-x_0|<R\}$ denote the open ball with radius $R>0$ centered at some $x_0 \in \mathbb{R}^n$. We denote the positive part of $a \in \mathbb{R}$ by $a_+:=\max\{a,0\}$.
 \bigskip

 In what follows, we denote by $C$, $C_{1}$, $C_{2}, \cdots $ different positive constants in a given context. Relevant dependencies on parameters will be emphasized using parentheses. For instance $C=C(n,p,\Omega,\cdots )$ means that $C$ depends on $n, p, \Omega \cdots$.  As customary, the equation number $(\,\cdot\,)_n$ denotes the $n$-th line of the Eq. $(\,\cdot\,)$.
 \bigskip

We next prepare some function spaces, defined on space-time region.  For two indices  $1 \leq p,q \leq \infty$, $L^{q}(t_1,t_2\,;\,L^{p}(\Omega))$ denotes the space of measurable real-valued functions on a space-time region $\Omega \times (t_1,t_2)$ with a finite norm
\[
\|v\|_{L^{q}(t_1,t_2\,;\,L^{p}(\Omega))}:=
\begin{cases}
\displaystyle \left(\int_{t_1}^{t_2}\|v(t)\|_{L^{p}(\Omega)}^{q}\,dt \right)^{1/q}\quad &(1 \leq q<\infty) \\
\displaystyle \esssup_{t_1 \leq t \leq t_2}\|v(t)\|_{L^{p}(\Omega)}\quad &(q=\infty),
\end{cases}\notag
\]
where 
\[
\|v(t)\|_{p}=\|v(t)\|_{L^{p}(\Omega)}:=
\begin{cases}
\left(\displaystyle \int_{\Omega}|v(x,t)|^{p}\,dx \right)^{1/p}\quad &(1 \leq p <\infty )\\
\esssup \limits_{x \in \Omega}|v(x,t)|\quad &(p=\infty).
\end{cases}
\]
When $p=q$, we write $L^p(\Omega \times (t_1,t_2))=L^{p}(t_1,t_2\,;\,L^{p}(\Omega))$ for brevity. For $1 \leq p <\infty$ the Sobolev space $W^{1,p}(\Omega)$ consists of measurable real-valued functions  in $\Omega$ which are weakly differentiable and of which weak derivatives are $p$-th integrable on $\Omega$, with the norm 
\[
\|v\|_{W^{1,p}(\Omega)}:=\left(\int_{\Omega}|v|^{p}+|\nabla v|^{p}\,dx \right)^{1/p}
\]
and let $W_{0}^{1,p}(\Omega)$ be the closure of $C_{0}^{\infty}(\Omega)$, the smooth functions with a compact support,  with respect to the norm $\|\cdot\|_{W^{1,p}}$.  The space $L^{q}(t_1,t_2\,;\,W_{0}^{1,p}(\Omega))$ comprises all measurable real-valued functions on space-time domain with a finite norm
\[
\|v\|_{L^{q}(t_1,t_2\,;\,W_{0}^{1,p}(\Omega))}:=\left(\int_{t_1}^{t_2}\|v(t)\|_{W^{1,p}(\Omega)}^{q} \,dt\right)^{1/q}.
\] 
Additionally, for an interval $I \subset \mathbb{R}$, the space $C(I ; L^{q}(\Omega))$  consists of all continuous functions $I \ni t \mapsto u (t) \in L^q (\Omega)$. The function space $C(I; W_{0}^{1,p}(\Omega))$ is defined analogously to the above. 
\smallskip

%
%
%

%
\medskip

We next need the following fundamental algebraic inequality, associated with the $p$ -Laplace operator (see \cite{Barrett-Liu, DiBenedetto1}).
\begin{lem}[Algebraic inequality]\label{Barrett-Liu}
For every $p \in (1,\infty)$ there exist positive constants $C_{1}(p,n)$ and $C_{2}(p,n)$ such that for all $\xi,\,\eta \in \mathbb{R}^{n}$
\begin{equation}\label{BL1}
||\xi|^{p-2}\xi-|\eta|^{p-2}\eta| \leq C_{1}(|\xi|+|\eta|)^{p-2}|\xi-\eta|
\end{equation}
and
\begin{equation}\label{BL2pre}
(|\xi|^{p-2}\xi-|\eta|^{p-2}\eta)\cdot (\xi-\eta) \geq C_2 (|\xi|+|\eta|)^{p-2}|\xi-\eta|^2,
\end{equation}
where dot $\cdot $ denotes the inner product in $\mathbb{R}^{n}$.  In particular, if $p\geq 2$, then 
\begin{equation}\label{BL2}
(|\xi|^{p-2}\xi-|\eta|^{p-2}\eta)\cdot (\xi-\eta) \geq C_2 |\xi-\eta|^p. 
\end{equation}
\end{lem}


Following \cite[Definition 3.2]{Kuusi-Misawa-Nakamura}, we present the definition of a weak solution to the $p$-Sobolev flow equation~\eqref{pS0}.

\begin{dfn}[weak solution of the $p$-Sobolev flow]\label{def of weak sol.}\normalfont 
Let $0<T \leq \infty$. A measurable function $u$ defined on $\Omega_{T}$ is called a \textit{weak solution} of~\eqref{pS0}
 if the following (D1)-(D4) are satisfied.
 \begin{enumerate}[(D1)]
 \item $u \in L^{\infty}(0,T\,;\,W^{1,p}(\Omega))$; \,\, $\partial_{t}(|u|^{q-1}u) \in L^{2}(\Omega_{T}).$
 
 \item There exists a function $\lambda (t) \in L^1_{\mathrm{loc}} (0, T)$ such that, for every $\varphi \in C^{\infty}_{0}(\Omega_T) $, 
 \begin{equation*}-\int_{\Omega_T}|u|^{q-1}u\varphi_{t}\,dz+\int_{\Omega_T}|\nabla u|^{p-2}\nabla u\cdot \nabla \varphi\,dz = \int_{\Omega_T} \lambda(t) |u|^{q-1}u\varphi dz.
 \end{equation*}

 \item $\|u (t)\|_{L^{q +1} (\Omega)} = 1$ for all positive $t<T$.

 \item $u = 0$ on $\partial \Omega \times (0,T)$
and $u (0) = u_0$ in $\Omega$ in the trace sense: \\
 $u(t) \in W^{1, p}_0 (\Omega)$ for almost every $t \in (0, T)$;
\[
\|u (t) - u_0\|_{W^{1,p} (\Omega)}\rightarrow 0 \quad \textrm{as}\quad  t\searrow 0.
\]
\end{enumerate}
\end{dfn}
\begin{prop}\label{lambda equality}
Let $u$ be a weak solution of~\eqref{pS0}.  If the initial value $u_{0}$ is nonnegative and bounded in $\Omega$ then so is $u$. Furthermore, the local $L^{1}$-function $\lambda(t)$ in the first line of~\eqref{pS0} is given by 
\[
\displaystyle \lambda(t)=\int_{\Omega}|\nabla u(t)|^{p}\,dx.
\]
\end{prop}
\begin{proof}
See \cite[Proposition 5.1, Proposition 5.2]{Kuusi-Misawa-Nakamura} for more details.
\end{proof}
In what follows, under the assumption that the initial data $u_0$ is in the Sobolev space $W^{1,p}_0 (\Omega)$, positive and bounded in $\Omega$ and satisfies $ \|u_0\|_{L^{q+1}(\Omega)}=1$, we address the following equation~\eqref{pS} in place of~\eqref{pS0}:

\begin{equation}\label{pS}
\begin{cases}
\,\,\partial_t(u^{q})-\Delta_pu=\lambda(t) u^{q}\quad &\textrm{in}\quad \Omega_\infty  \\
\,\,\displaystyle \| u(t) \|_{L^{q+1}(\Omega)} =1\quad &\textrm{for all}\,\, t \geq 0\\[1mm]
\,\,u=0\quad &\textrm{on}\quad \!\!\partial\Omega \times (0,\infty) \\
\,\,u(\cdot, 0)=u_0(\cdot ) \quad &\textrm{in}\quad \Omega.

\end{cases}
\end{equation}
%

\section{Prototype Doubly Nonlinear Equation}\label{Sec. Prototype doubly nonlinear equation}
In this section, we study the nonlinear scaling for the following doubly nonlinear equation~\eqref{bpS}.
Firstly, we recall the definition and the global existence result of weak solutions to~\eqref{bpS}:
\begin{dfn}\label{def of v}\normalfont 
Let $0 < S \le \infty$. A measurable function $v$, defined on $\Omega_{S}:=\Omega \times (0,S)$,  is a \textit{weak supersolution (subsolution)} of~\eqref{bpS} if the following conditions are satisfied:
\begin{enumerate}[(i)]
\item $v \in L^{\infty}(0,S\,;\,W^{1,p}(\Omega))$,\,\,\,$\partial_{s} (|v|^{q-1}v) \in L^{2}(\Omega_{S})$.
\item For every nonnegative $\eta \in C^{\infty}_{0}(\Omega_{S})$
\[
-\int_{\Omega_{S}}|v|^{q-1}v \cdot \partial_{s}\eta\,dxds+\int_{\Omega_{S}}|\nabla v|^{p-2}\nabla v\cdot \nabla \eta\,dxds \geq (\leq)0.
\] 
\item $v = 0 $ on $\partial \Omega \times (0,S)$ and $v(0) = v_0$ in $\Omega$ in the trace sense: \\
 $v(s) \in W^{1, p}_0 (\Omega)$ for almost every $s \in (0, S)$;
\[
\|v(s) - v_0\|_{W^{1,p} (\Omega)}\rightarrow 0 \quad \textrm{as}\quad  s\searrow 0.
\]

\end{enumerate}
\smallskip
Furthermore, a measurable function $v$ defined on $\Omega_{S}$ is called a \emph{weak solution } to~\eqref{bpS} if it is simultaneously a weak super and subsolution, i.e., 
\[
-\int_{\Omega_{S}}|v|^{q-1}v \cdot \partial_{s}\eta\,dxds+\int_{\Omega_{S}}|\nabla v|^{p-2}\nabla v\cdot \nabla \eta\,dxds=0.
\] 
holds for every $\eta \in C^{\infty}_{0}(\Omega_{S})$.
\end{dfn}
\medskip
%

%
In \cite[Theorem 1.1]{Nakamura-Misawa} and  \cite{Kuusi-Misawa-Nakamura},  we proved the global existence of a weak solution to~\eqref{bpS} and it's regularity estimates as follows:
\begin{thm}
[Global existence of~\eqref{bpS} \cite{Kuusi-Misawa-Nakamura, Nakamura-Misawa}]\label{our theorem}
Assume the initial value $v_0$ be in $W^{1, p}_0 (\Omega)$, nonnegative and bounded in $\Omega$. 
Then there exists a global in time weak solution  $v$ of~\eqref{bpS}, which is nonnegative and bounded in $\Omega_\infty$ that is, 
\begin{equation}\label{MP}
0 \leq v\leq \|v_{0}\|_{\infty} \quad \textrm{in}\,\,\Omega_\infty.
\end{equation}
In addition, $v$ satisfies the following energy equality, for $0\leq s_{1}<s_{2}< \infty$,
\begin{equation}\label{energy eq1 of v}
\|v(s_{2})\|_{q+1}^{q+1}+\frac{q+1}{q}\int_{s_{1}}^{s_{2}}\|\nabla v(s)\|_{p}^{p}\,ds=\|v(s_{1})\|_{q+1}^{q+1}
\end{equation}
and, the integral inequalities hold true, for any nonnegative $s<\infty$, %
\begin{equation}\label{energy ineq1 of v}
\|v(s)\|_{q+1} \leq \|v_{0}\|_{q+1},
\end{equation}
\begin{equation}\label{energy ineq2 of v}
\|\nabla v(s)\|_{p} \leq \|\nabla v_{0}\|_{p},
\end{equation}
\begin{equation}\label{energy ineq3 of v}
\|\partial_s v^q\|_2^2
\le C \|v_0\|_\infty^{q-1} \|\nabla v_0\|_p^p,
\end{equation}
where $C=C(n,p)$ is a positive constant, and the $L^p$-norm on space of $v$ is denoted by $\|v\|_{p}:=\|v\|_{L^{p}(\Omega)}$ for brevity.
\end{thm}
\begin{proof}
Eq.~\eqref{MP} follows from~\cite[Propositions 3.4, 3.5]{Kuusi-Misawa-Nakamura}. Using a similar argument to~\cite[Appendix B]{Kuusi-Misawa-Nakamura}, one can prove~\eqref{energy eq1 of v} and from this,~\eqref{energy ineq1 of v} immediately follows.
By the same way as in~\cite[Lemma 3.2, (3.7); Lemma 4.1; Proof of Theorem 1.1]{Nakamura-Misawa},~\eqref{energy ineq2 of v}  is actually verified. Finally,~\eqref{energy ineq3 of v} is proved via~\cite[Lemmas 3.4, 4.1]{Kuusi-Misawa-Nakamura}.
\end{proof}

\begin{rmk}\normalfont
The existence of a weak solution $v$ to~\eqref{bpS} in Theorem~\ref{our theorem} is proved by a time discretization and weak convergence of bounded approximating solutions in a reflexive Banach space. In our preceding work \cite{Nakamura-Misawa}, we showed the existence on $[0, S]$ of a weak solution for any positive $S<\infty$. As seen from the proof~\cite[Section 5, pp.167--168]{Nakamura-Misawa}, we can choose $S=+\infty$.
\end{rmk}
%

\subsection{Extinction of Solutions}

We study the finite-time extinction of a solution to~\eqref{bpS}. Firstly, the extinction time is defined as follows:
\begin{dfn}\label{def. of the  extinction time}\normalfont
Let $v$ be a nonnegative weak solution to~\eqref{bpS} in $\Omega_{\infty}$. We call a positive number $S^{\ast}$ as the \emph{extinction time} of $v$ provided that the following conditions hold:
\begin{enumerate}[(i)]
\item $v(x,s)$ is nonnegative and not identically zero on $\Omega \times (0,S^{\ast})$ 
\item $v(x,s)=0$ for any $x \in \overline{\Omega}$ and all $s \geq S^{\ast}$.
\end{enumerate}
\end{dfn}
Now we will show a finite time extinction of a solution of~\eqref{bpS}. 
For this purpose we apply the comparison theorem \cite[Theorem 3.6]{Kuusi-Misawa-Nakamura}, which is originally provided by Alt-Luckhaus (\cite[Theorem 2.2, p.325]{Alt-Luckhaus}).
\begin{thm}[Comparison theorem]\label{Comparison theorem}
Let $0 < S \leq \infty$ and, let $v_{1}$ and $v_{2}$ be a weak supersolution and subsolution to~\eqref{bpS} in $\Omega_S$, respectively. 
Suppose that $v_{1} \geq v_{2}$ on $\partial_{p}\Omega_{S}$.
Then 
\begin{equation*}
v_{1} \geq v_{2}\quad in \quad \Omega_S.
\end{equation*}
\end{thm}
For the construction of an appropriate comparison function we use a special solution to the elliptic type equation associated with~\eqref{bpS}. This special solution is called \emph{Talenti function} \cite{Talenti}, defined as
\begin{equation}\label{e.Talenti}
Y_{a,b,y}(x):= \left(a + b |x-y|^{\frac{p}{p-1}} \right)^{-\frac{n-p}{p}}, \quad x, y \in \mathbb{R}^{n},
\end{equation}
where $a$ and $b$ are positive numbers.
In his seminal paper~\cite{Talenti}, Talenti showed that this function realizes the best constant in the Sobolev inequality. Moreover, a direct computation shows that $Y_{a,b}$ solves the equation
\begin{equation*} 
-\Delta_p Y_{a,b,z} = n\left(\tfrac{n-p}{p-1} \right)^{p-1} a b^{p-1} Y_{a,b,z}^{q} \quad \mbox{in } \mathbb{R}^n.
\end{equation*}
As Sciunzi showed in~\cite{Sciunzi}, there is a one-parameter family of functions classifying the solutions, up to translations, of $-\Delta_p Y_\lambda = Y_\lambda^{q}$. Indeed, one chooses $a$ and $b$ so that $n\left(\tfrac{n-p}{p-1} \right)^{p-1} a b^{p-1} = 1$ and then, by \cite{Sciunzi}, the solution of $-\Delta_p Y = Y^{q}$ is necessarily of the form
\begin{equation}  \label{e.Talentispecial}
Y(x) = Y_{\lambda, y}(x) := \frac1{\lambda} \left( n \left( \frac{n-p}{p-1} \right)^{p-1} \right)^{\frac 1p}  \left( 1 + \left( \frac{|x-y|}{\lambda}\right)^{ \frac{p}{p-1}}
\right)^{-\frac{n-p}{p}}
\end{equation}
with a parameter $\lambda>0$.
\bigskip
\noindent
We  show that a solution of~\eqref{bpS} vanishes in finite time.
\begin{prop}[Finite time extinction of solutions]\label{Finite time extinction of solutions}
Let $v$ be a nonnegative  weak solution to~\eqref{bpS} in $\Omega_{\infty}$. Then there exists a extinction time $S^{\ast} > 0$ for $v$, which is bounded as follows :
\begin{equation*}
S^{\ast} \leq \frac{q}{q+1-p} \Bigg( \frac{\max\limits_\Omega u_0}{\min\limits_\Omega Y}\Bigg)^{q+1-p},
\end{equation*}
where $Y$ is Talenti's function defined by~\eqref{e.Talentispecial}. 
\end{prop}

\begin{proof} By translation, we may assume the origin $0 \in \Omega$. Let $v=v(x,s)$ be a nonnegative weak solution to~\eqref{bpS}. By the nonnegativity $v$ is a weak solution of\begin{equation*}
\partial_{s}v^{q}-\Delta_{p}v = 0 \quad \textrm{in} \quad \Omega_{\infty}.
\end{equation*}
Next, let $W(x,s) = X(x)Z(s)$ be a nonnegative separable solution of 
\begin{equation*}
\partial_{s} W^{q}-\Delta_{p} W =0\quad \textrm{in}\,\,\mathbb{R}^n \times (0,\infty).
\end{equation*}
Then $X(x)Z(s)$ satisfies
\begin{eqnarray}\label{vanishing time eq1}
\begin{cases}
(Z(s)^q)'=\mu Z(s)^{p-1}\quad &\textrm{in}\,\, (0,\infty) \\
\Delta_pX=\mu X^q \quad &\textrm{in}\,\,\mathbb{R}^n,
\end{cases}
\end{eqnarray}
where $\mu$ is a separation constant. By an integration by parts we see that $\mu<0$. Set~$X:=(-\mu)^{-\frac{1}{q+1-p}}Y$ to obtain
\begin{equation}\label{vanishing time eq2}
-\Delta_p Y=Y^q \quad \textrm{in}\,\,\mathbb{R}^n.
\end{equation}
As discussed before the proof, an energy-finite solution to (\ref{vanishing time eq2}) is given by~\eqref{e.Talentispecial}.
By a straightforward computation, we find that
\begin{equation*}
Z(s)=Z(0) \bigg(1+\mu\frac{q+1-p}{q}Z(0)^{p-(q+1)}s \bigg)_+^\frac{1}{q+1-p}
\end{equation*}
solves the first equation in (\ref{vanishing time eq1}), where $Z(0)$ is the initial data. Thus the vanishing time $Z_0$ of $Z(s)$ is given by
\begin{equation*}
Z_0=(-\mu)^{-1}\frac{q}{q+1-p}Z(0)^{q+1-p}.
\end{equation*}
Let $V(x,s)$ be $(-\mu)^{-\frac{1}{q+1-p}}Y(x)Z(s)$. Then
\begin{equation*}
0=v(x,s) \leq V(x,s) \quad \textrm{on}\,\,\,\partial \Omega \times [0, \infty)
\end{equation*}
We choose the initial data for the ODE in~\eqref{vanishing time eq1} as 
\begin{equation}\label{choice of T(0)1}
Z(0)=\frac{\max\limits_\Omega u_0}{\min\limits_\Omega Y}(-\mu)^\frac{1}{q+1-p}
\end{equation}
and therefore, we find that
\begin{equation*}
u_0(x) \leq V(x,0)\quad \textrm{in}\,\,\,\Omega.
\end{equation*}
According to the comparison theorem \cite[Theorem 3.6]{Kuusi-Misawa-Nakamura}, we have
\begin{equation*}
v(x,s) \leq V(x,s) \quad \textrm{in}\,\,\,\Omega_S\,\,\,\textrm{for any positive}\,\,\,S < \infty
\end{equation*}
and thus, the vanishing time $S^{\ast}$ of $v(x,s)$ is estimated as
\begin{equation*}
S^{\ast} \leq S_0=\frac{q}{q+1-p} \Bigg( \frac{\max\limits_\Omega u_0}{\min\limits_\Omega Y}\Bigg)^{q+1-p},
\end{equation*}
where (\ref{choice of T(0)1}) is used. The proof is complete.
\end{proof}
%
%

\section{Nonlinear Intrinsic Scaling Transformation}\label{Sec. Nonlinear intrinsic scaling transformation}

In this section we will introduce a scaling transformation, which transforms the prototype equation~\eqref{bpS} into the $p$-Sobolev equation~\eqref{pS}. 
Hereafter we choose the initial data $v_0$ in~\eqref{bpS} as $u_0$ in~\eqref{pS}. We suppose that the initial data $u_0$ is in $W^{1,p}_0 (\Omega)$, positive and bounded in $\Omega$, and $\|u_0\|_{q + 1}=1$.  As in Theorem \ref{our theorem}, by \cite[Proposition 3.4]{Kuusi-Misawa-Nakamura}, the solution $v$ of~\eqref{bpS} must be nonnegative and thus, we can consider~\eqref{bpS} as %
\begin{equation}\label{bpS+}
\begin{cases}
\,\,\partial_s(v^{q})-\Delta_pv=0\quad &\textrm{in}\quad \Omega_S  \\ 
\,\,v=0\quad &\textrm{on}\quad \!\!\partial\Omega \times (0,S) \\
\,\,v(\cdot, 0)=u_0(\cdot ) \quad &\textrm{in}\quad \Omega.
\end{cases}
\end{equation}
From now on, we will mainly consider~\eqref{bpS+} instead of~\eqref{bpS}.
\medskip

%
%
%
Let us consider the following nonlinear intrinsic scaling transforming~\eqref{bpS+} to~\eqref{pS}.
\begin{prop}[Nonlinear intrinsic scaling]\label{Nonlinear intrinsic scaling}
Let $v$ be a nonnegative weak solution to the equation~\eqref{bpS+} in $\Omega_{\infty}$ and let $S^\ast<+\infty$ be a finite extinction time of $v$. 
There exist unique $\Lambda \in C^1[0,\infty)$ solving
\begin{equation}\label{def. of Lambda}
\begin{cases}
\Lambda^\prime(\tau) = (S^\ast)^{-1}\left( \,\displaystyle \int_{\Omega} v^{q+1}\left(x,S^\ast\left(1-e^{-\Lambda(\tau)} \right)\right) \, dx \right)^{\frac{p}{n}}\\[5mm]
\Lambda(0) = 0
\end{cases}
\end{equation}
and, subsequently, $g \in C^1[0,\infty)$ solving
\begin{equation}\label{def. of g}
\begin{cases}
g^\prime(t) = e^{\Lambda(g(t))} \\
g(0) = 0
\end{cases}
\end{equation}
such that the following is valid: Let 
\begin{equation}\label{def. of t}
s(t)=S^\ast \left(1-e^{-\Lambda(g(t))}\right)
\end{equation}
%
%
%
and set
\begin{equation}\label{def. of u}
u(x,t) := \frac{v\left(x,s(t)\right)}{\gamma(t)}, \quad \gamma(t) := \left(\,\displaystyle \int_{\Omega} v^{q+1}\left(x,s(t)\right) \, dx \right)^{\frac{1}{q+1}}.
\end{equation}
Then $u$ is a nonnegative weak solution of the $p$-Sobolev flow~\eqref{pS} on $\Omega_\infty$. More precisely, $u$ satisfies the conditions (D1)--(D4) of Definition~\ref{def of weak sol.}, where
$\displaystyle \lambda(t) := -q \frac{\gamma'(t)}{\gamma(t)}=\int_{\Omega}|\nabla u(x,t)|^{p}\,dx$.
\end{prop}

%
%
%
\begin{proof}[\normalfont \textbf{Proof of Proposition~\ref{Nonlinear intrinsic scaling}}]
Here, we will make a formal computation and show the relevance of intrinsic scaling above to the $p$-Sobolev flow. The rigorous argument will be given in Appendix~\ref{Sec. rigorous argument}.
\smallskip

Firstly, let us verify that $u$ satisfies~\eqref{pS}$_1$.
Noticing~\eqref{def. of t}
\[
s(t)=S^\ast(1-e^{-\Lambda(g(t))}) \quad \iff \quad \Lambda(g(t))=\log\left(\frac{S^\ast}{S^\ast-s(t)} \right) ,
\]
we compute as
\begin{align*}
\frac{d}{dt}\Lambda(g(t))&=\Lambda^\prime(g(t))g^\prime(t) \\[2mm]
&=(S^\ast)^{-1}e^{\Lambda(g(t))}\left(\,\,\displaystyle \int_{\Omega} v^{q+1}(x,s(t)) \, dx \right)^{\frac{p}{n}}\\[2mm]
&=(S^\ast)^{-1}e^{\Lambda(g(t))}\gamma(t)^{(q+1)\frac{p}{n}}
\end{align*}
and thus,
\begin{equation}\label{nis eq.0}
s_t=\frac{ds}{dt}=S^\ast e^{-\Lambda(g(t))}\frac{d}{dt}\Lambda(g(t))=\gamma(t)^{(q+1)\frac{p}{n}}.
\end{equation}
By~\eqref{energy eq1 of v} and~\eqref{energy ineq2 of v} in Theorem~\ref{our theorem}, $\displaystyle s \mapsto \left(\int_\Omega v^{q+1}(x,s)\,dx\right)^{\frac{1}{q+1}}$ is Lipschitz continuous.
This together with~\eqref{nis eq.0} provides 
\begin{align}\label{nis eq.1}
\partial_{t}u^q&=\partial_sv^qs_t\gamma^{-q}+v^q(-q)\gamma^{-q-1}\gamma^\prime(t) \notag\\
&=\partial_sv^{q}\gamma^{-q+(q+1)\frac{p}{n}}-qu^{q}\gamma^{-1}\gamma^\prime(t). 
\end{align}
Multiplying~\eqref{pS}$_1$ by $v$ and integration by parts give 
\begin{equation}\label{nis eq.2}
\frac{q}{q+1}\frac{d}{ds}\int_{\Omega} v^{q+1}(s)\,dx+\int_{\Omega}|\nabla v(s)|^{p}\,dx=0,
\end{equation}
that is the same reasoning as~\eqref{energy eq1 of v} in Theorem~\ref{our theorem}.
Furthermore
\begin{align} \label{nis eq.3}
\gamma^\prime(t)&=\dfrac{1}{q+1} \left(\,\,\displaystyle \int_{\Omega} v^{q+1}(x,s(t)) \, dx \right)^{\frac{1}{q+1}-1}\frac{d}{ds} \int_\Omega v^{q+1}(s)\,dx \Bigg|_{s=s(t)}s_t \notag \\[2mm]
&\!\!\stackrel{\eqref{nis eq.2}}{=}-\frac{1}{q}\gamma^{1-(q+1)}\int_\Omega|\nabla v(s(t))|^{p}\,dx\cdot s_t\notag\\[2mm]
&=-\frac{1}{q}\gamma\int_\Omega|\nabla u(t)|^{p}\,dx
\end{align}
and
\begin{equation}\label{nis eq.4}
\Delta_pu=\gamma^{1-p}\Delta_pv.
\end{equation}
From~\eqref{nis eq.1} and~\eqref{nis eq.3} it follows that
\begin{align}\label{nis eq.5}
\partial_tu^q&=\partial_sv^{q}\gamma^{1-p}+\left(\,\,\int_{\Omega}|\nabla u(t)|^{p}\,dx\right)u^q \\
&=\Bigg[\partial_sv^q+\left(\,\,\int_{\Omega}|\nabla v(s(t))|^{p}\,dx \right)v^q\gamma^{-(q+1)} \Bigg]\gamma^{1-p}.\label{nis eq.5'}
\end{align}
Eq.~\eqref{nis eq.5} together with~\eqref{nis eq.4} and Eq.~\eqref{pS}$_1$ yield that 
\begin{equation*}
\partial_tu^q-\Delta_pu=\left(\,\,\int_{\Omega}|\nabla u(t)|^p\,dx\right)u^q,
\end{equation*}
which is exatcly~\eqref{pS}$_1$.
\medskip

We will verify that $u$ satisfies the condition (D1) in Definition~\ref{def of weak sol.}. Let $t_0<\infty$ be any positive number  and set $s_0=S^\ast(1-e^{-\Lambda(g(t_0))})$. We shall notice the fact : As shown later in~\eqref{c0} in Lemma~\ref{bounded from below lem}, we find that there is a positive number $c_0$ such that $c_0:=\min \limits_{0 \leq s \leq s_0}\|v(s)\|_{q+1}>0$. 
From $\gamma (t) \geq c_0 > 0$ and~\eqref{energy ineq2 of v}, it follows that 
\begin{align}\label{lambda bounds}
\int_{\Omega}|\nabla u(t)|^p\,dx&=\frac{1}{\gamma(t)^p}\int_{\Omega}|\nabla v(s(t))|^p\,dx \notag\\
& \leq c_0^{-p}\|\nabla u_{0}\|_p^p <\infty
\end{align}
and thus, $u \in L^{\infty}(0,t_0\,;\,W^{1,p}(\Omega))$. By changing of variable $s=S^\ast(1-e^{-\Lambda(g(t))})$ and~\eqref{nis eq.0},
and merging~\eqref{nis eq.5'},\eqref{c0} in Lemma~\ref{bounded from below lem},~\eqref{MP},~\eqref{energy ineq1 of v},~\eqref{energy ineq2 of v} and~\eqref{energy ineq3 of v} we get
\begin{align*}
&\int_0^{t_0}\!\!\!\int_\Omega(\partial_t u^q)^2\,dxdt=\int_0^{s_0}\!\!\! \int_\Omega(\partial_t u^q)^2\gamma^{-(q+1)\frac{p}{n}}\,dxds\notag\\
&\leq 2\int_0^{s_0}\!\!\!\int_{\Omega}\Bigg[\left(\partial_sv^q\right)^2+\left\|\nabla v(s(t))\right\|_p^{2p}v^{2q}\gamma^{-2(q+1)} \Bigg]\gamma^{2(1-p)} \cdot\gamma^{-(q+1)\frac{p}{n}}\,dxds\notag \\
&\leq 2\|u_0\|_{q+1}^{(q+1)\frac{p}{n}} c_0^{-2q} \int_{0}^{s_0}\!\!\!\int_{\Omega}(\partial_s v^q)^{2}\,dxds \notag\\
&\quad \quad +2\|u_{0}\|_{q+1}^{(q+1)\frac{p}{n}} c_0^{-4q-2} \|u_{0}\|_{\infty}^{2q}\|\nabla u_{0}\|_{p}^{2p}|\Omega|s_0<\infty,
\end{align*}
which yields (D1) for any positive $T < \infty$.
\smallskip

By the very definition of $u$ as in~\eqref{def. of u}, 
$\displaystyle \int_{\Omega} u^{q+1}(x,t)\, dx = 1$ for any $t \in [0,\infty)$, that is  (D3) with $T=\infty$. 

\smallskip

Since  $v(s) \in W_{0}^{1,p}(\Omega)$ for a.e. $s>0$, $u(t)=v(s(t))/\gamma(t) \in W_{0}^{1,p}(\Omega)$ for a.e. $t \in [0,t_0]$.
In addition, 
\begin{align}\label{nis eq.6}
\|u(t)-u_0\|_{W^{1,p}(\Omega)}&=\left\| \frac{v(s(t))}{\gamma(t)} -u_0\right\|_{W^{1,p}(\Omega)} \notag\\
&\leq \frac{1}{\gamma(t)}\bigg\{\left\|v(s(t)) -u_{0}\right\|_{W^{1,p}(\Omega)}+\|u_0\|_{W^{1,p}(\Omega)} |\gamma(t)-\gamma(0)|\bigg\},
\end{align}
where $\gamma(0)=\|u_0\|_{q+1}=1$. Remark that $s(t)=S^\ast(1-e^{-\Lambda(g(t))}) \searrow 0 \iff t \searrow 0$ because $\Lambda(\tau)$, $0 \leq \tau <\infty$ and $g(t)$, $0 \leq t <\infty$, are monotone increasing and $\Lambda (0)=g(0)=0$ by~\eqref{def. of g} and~\eqref{def. of Lambda}. By the Minkowski and Sobolev inequalities, we get, as $t \searrow 0$,
\begin{align}\label{nis eq.7}
|\gamma(t)-\gamma(0)| &\leq \bigg|\left\|v(s(t))\right\|_{q+1}-\|u(0)\|_{q+1}\bigg| \notag\\
&\leq C\left\|v(s(t))-u_{0}\right\|_{W^{1,p}(\Omega)} \to 0
\end{align}
since $\|v(s)-u_{0}\|_{W^{1,p}(\Omega)} \to 0$ as $s \searrow 0$. Merging $\gamma (t) \geq c_0 > 0$,~\eqref{nis eq.6} and~\eqref{nis eq.7}, we obtain $\|u(t)-u_{0}\|_{W^{1,p}(\Omega)}$ as $t \searrow 0$, that gives (D4). 
\medskip

Therefore we finish the proof.
\end{proof}

%

\section{Proof of Theorem~\ref{mainthm}} \label{Sec. Proof of mainthm}
In this section, we shall prove our main theorem,Theorem~\ref{mainthm}. 
\medskip

The scheme of our proof is the following: Firstly, by Theorem~\ref{our theorem}, we will solve the prototype equation~\eqref{bpS} with the initial data $u_0$ and then, by Proposition~\ref{Nonlinear intrinsic scaling}, we transform the solution $v$ to the desired solution $u$ of the $p$-Sobolev flow~\eqref{pS}, that can be possible up to any finite time, since the extinction time of solution $v$ is converted to the infinity. 
Here, the expansion of positivity of the solution of~\eqref{pS} on the domain is used. In particular, the time-length of expansion of positivity is estimated only by the volume, $\|u_0\|_{L^{q+1}(\Omega)}=1$, the boundedness, and the positive lower bound of the initial data $u_0$  in the interior of domain.
See Proposition~\ref{Interior positivity by the volume constraint} for details.
The solution $u$ is actually bounded at any finite time as in~\eqref{eq. Boundedness of pS} of Proposition~\ref{Boundedness of pS}. In this way, we have the global existence of solution of the $p$-Sobolev flow~\eqref{pS}.
\bigskip

We have the boundedness of weak solutions of $p$-Sobolev flow~\eqref{pS}. Here we use by the fact that by Proposition~\ref{lambda equality}  we have that $\lambda (t) = \|\nabla u (t)\|_{p}^p$ in~\eqref{pS}.
\begin{prop}[Boundedness of the $p$-Sobolev flow]\label{Boundedness of pS}
Let $u$ be a nonnegative weak solution of~\eqref{pS} in $\Omega_{T}$. Then $u$ is bounded from above in $\Omega_{T}$ and 
\begin{equation}\label{eq. Boundedness of pS}
\|u(t)\|_{\infty} \leq e^{\frac{1}{q}\int_0^T\|\nabla u(t)\|_{p}^{p}\,dt}\|u_{0}\|_{\infty} \quad \textrm{for every} \quad 0\leq t < T.
\end{equation}
\end{prop}
%

\bigskip

In \cite{Kuusi-Misawa-Nakamura} we proved \emph{the expansion of positivity} of a solution of the doubly nonlinear equation such as~\eqref{pS} and~\eqref{bpS}. In particular, the convexity of domain is not needed by virtue of the so-called Harnack chain argument. See \cite[Theorem 4.7, Corollary 4.8]{Kuusi-Misawa-Nakamura} for detail and its proof.  We are going to deduce the refined assertion of them.
\medskip

Before stating, we set the notation as below. Let $\Omega^{\prime}$ be a subdomain contained compactly in $\Omega$. Let $\rho$ be any positive number satisfying $\rho \leq \frac{1}{16}\mathrm{dist} (\Omega^{\prime}, \partial \Omega)$. 
\begin{figure}
\begin{center}
\begin{tikzpicture}[scale=0.50]
            \draw  plot[draw=black, smooth, tension=.8] coordinates {(-4.0,0.55) (-3.3,2.8) (-1.2,4.0) (1.8,3.6) (4.4,3.8) (5.2,2.8) (5.3,0.6) (2.8,-2.5) (0,-1.0) (-3.3,-1.5) (-4.0,0.6)};
            \filldraw[pattern=north east lines, pattern color=blue, thick]  plot[smooth, tension=.8] coordinates {(-3.5,0.5) (-3,2.5) (-1,3.5) (1.5,3) (4,3.5) (5,2.5) (5,0.5) (2.5,-2) (0,-0.5) (-3,-1.0) (-3.5,0.5)};
            \draw[<->] (0,3.9)--(0,3.4);
    \draw (0.05,3.42) node[right]{\footnotesize $16\rho$};
            \draw plot[smooth, tension=.7] coordinates {(-3.5,0.5) (-3,2.5) (-1,3.5) (1.5,3) (4,3.5) (5,2.5) (5,0.5) (2.5,-2) (0,-0.5) (-3,-1.0) (-3.5,0.5)};
           \draw (1,1.5)node[below]{\small $\Omega'$};
           \draw (5.7,1.5)node[right]{\small $\Omega$};
          \end{tikzpicture}
 \end{center}
 \caption{Domain and subdomain}
 \end{figure}
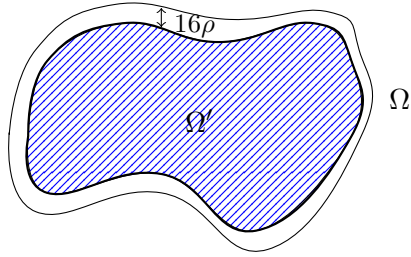

Now, we state the refinement of expansion of positivity with a waiting time (cf. \cite[Theorem 4.9]{Kuusi-Misawa-Nakamura}).  

\begin{thm}[Expansion of positivity with a waiting time]\label{ex. thm prime}
Let $u$ be a nonnegative weak solution of~\eqref{pS} in $\Omega_T$. Let $\Omega^\prime$ be a subdomain contained compactly in $\Omega$. Let $\rho$ be any positive number satisfying $\rho \leq \mathrm{dist} (\Omega^\prime, \partial \Omega)/16$.  Let $t_0 \in (0, T]$. Suppose that
\begin{equation}\label{ex. propassumption prime}
\big|\Omega^\prime \cap \{u(t_0) \geq L\} \big| \geq \alpha |\Omega^\prime|
\end{equation}
holds for some $L>0$ and $\alpha \in (0,1]$. Then there exist positive integer $N=N(\Omega^\prime, \rho)$, positive real number families $\delta_{0}, \delta_{N},\,\eta_{N}, \eta_{N+1}, \sigma_{N} \in (0,1),\,J_{N},\,I_{N}\in \mathbb{N}$ depending on $\alpha, N, n, p$ and independent of $L$, and a time $t_N>t_0$ such that
\begin{equation*}
u \geq \eta_{N+1} L
\end{equation*}
almost everywhere in
\begin{equation*} 
\Omega^{\prime}\times \left(t_{N}+ (1-\sigma_N) \delta_N (\eta_N L)^{q+1-p} \rho^p,\,t_{N}+\delta_N (\eta_N L)^{q+1-p} \rho^p \right), 
\end{equation*}
where $\sigma_N=e^{-(\tau_N + 2e^{\tau_N})}$, $e^{\tau_N} = C 2^{I_N+J_N}$
with $C=C(n,p)>0$, and $t_{N}$ is written as
\begin{equation*}
t_N=t_0+(\delta_0-\delta_N\eta_N^{q+1-p})L^{q+1-p}\rho^p
\end{equation*}
and thus, the terminal time of the time interval above is
\begin{equation*}
t_{0}+\delta_{0}L^{q+1-p}\rho^p.
\end{equation*}
\end{thm}
\begin{proof}
The proof of this theorem is postponed, and will be given in Appendix~\ref{two proofs}.

\end{proof}
We also state the refined expansion of positivity without a waiting time (cf. \cite[Corollary 4.10]{Kuusi-Misawa-Nakamura}).
\begin{prop}[Expansion of positivity without a waiting time]\label{ex. cor prime}
Let $u$ be a nonnegative weak solution of~\eqref{pS} in $\Omega_T$. Let $\Omega^\prime$ be a subdomain contained compactly in $\Omega$. Let $\rho$ be any positive number satisfying $\rho \leq \mathrm{dist} (\Omega^\prime, \partial \Omega)/16$. Suppose that $u(t_0)>0$ in $\Omega$ for some  $t_0 \in [0,T)$. Then there exist positive numbers $\eta_0$ and  $\tau_0$ such that
\begin{equation*}
u \geq \eta_0 \quad \textrm{a.e.}\quad  \textrm{in}\quad  \Omega^\prime \times (t_0, t_0+\tau_0).
\end{equation*}
%
%
\end{prop}
We also give the proof of this proposition in Appendix~\ref{two proofs}.
\medskip

Applying Theorem~\ref{ex. thm prime} and Proposition~\ref{ex. cor prime} with $t_0 = 0$, we have another refinement of the interior positivity by the volume constraint (cf. \cite[Proposition 5.4]{Kuusi-Misawa-Nakamura}). 
\begin{prop}[Interior positivity by the volume constraint]\label{Interior positivity by the volume constraint}
Let the initial data $u_0 \in W^{1,p}_{0}(\Omega)$ be positive, bounded in $\Omega$ and satisfy $\|u_{0}\|_{q+1}=1$. Let $u$ be a nonnegative weak solution of~\eqref{pS}  in $\Omega_T$ with $T>0$. Put $M:=e^{\frac{1}{q}\int_0^T\|\nabla u(t)\|_{p}^{p}\,dt} \|u_{0}\|_{\infty}$ and let $\Omega^{\prime}$ be a subdomain compactly contained in $\Omega$ satisfying $|\Omega \setminus \Omega^{\prime}| \leq \frac{1}{4M^{q+1}}$. Then there exists a positive constant $\eta$ such that %
\[
u(x,t)\geq \eta L \quad \textrm{in}\quad \Omega^{\prime} \times [0,T].
\]
Here $0<L\leq \min \left\{ \left(\frac{1}{4|\Omega^{\prime}|}\right)^{\frac{1}{q+1}}, \inf \limits_{\Omega^{\prime\prime}} u_{0} \right\}$, 
  where $\Omega^{\prime\prime}$ is compactly contained in $\Omega$ and compactly containing $\Omega^\prime$, and  the positive constant $\eta$ depends only on $p, n, \Omega^{\prime}, M$ and $N$, where $N$ is the number of chain balls of $\Omega^{\prime}$. The constant $\eta$ is given as a non-increasing positive function in $T$ and $\|u_{0}\|_{\infty}$.
\end{prop}
We present the proof of this proposition in Appendix~\ref{another proof}.
\bigskip

Finally, we give the positivity near the boundary for the solutions of $p$-Sobolev flow~\eqref{pS}.  See \cite[Propositions 5.5, 4.9]{Kuusi-Misawa-Nakamura} for details.

\begin{prop}[Positivity around the boundary]\label{positivity around the boundary}
Suppose that $u_0>0$ in $\Omega$. Let $u$ be a nonnegative weak solution $u$ to~\eqref{pS} in $\Omega_T$. Then $u$ is positive in $\Omega_T$ near the boundary.
\end{prop}

Under the above preliminaries, we now prove Theorem~\ref{mainthm}.

\begin{proof}[\normalfont \textbf{Proof of Theorem~\ref{mainthm}}]
We divide the proof  in three steps.
\bigskip

\textbf{Step 1:}
\quad We choose the initial data $v_0$ as $u_0$, that of the $p$-Sobolev flow~\eqref{pS} and solve the prototype equation~\eqref{bpS} with the initial data $v_0 = u_0$. Let $v$ be a nonnegative weak solution of~\eqref{bpS} in $\Omega_{\infty}$ with the initial data $v_0 = u_0$. Let $S^{\ast}$ be the extinction time of $v$. 
%
\bigskip

\textbf{Step 2:}
\quad Here we will verify the positivity for the $p$-Sobolev flow~\eqref{pS}. From Propositions~\ref{Interior positivity c0 by the volume constraint} and~\ref{positivity around the boundary}, for any positive $T <\infty$
\begin{equation}\label{note1}
u >0 \quad \mathrm{in}\quad \Omega\times [0,T].
\end{equation}

From the global existence result for the prototype doubly nonlinear equation~\eqref{bpS} in Theorem~\ref{our theorem} and the nonlinear intrinsic scaling transformation in Proposition~\ref{Nonlinear intrinsic scaling},  we plainly get the global existence for the $p$-Sobolev  flow~\eqref{pS}.
\bigskip

\textbf{Step 3:}
\quad Finally, we will show the local H\"{o}lder regularity for the $p$-Sobolev flow~\eqref{pS}.
\medskip

Following \cite[Section 5.2]{Kuusi-Misawa-Nakamura}, we recall the result of H\"older and gradient H\"{o}lder continuity of the solution to $p$-Sobolev flow~\eqref{pS} with respect to space-time variable.
\medskip

Suppose the initial value $u_0 >0$ in $\Omega$. Then by Propositions~\ref{Interior positivity by the volume constraint} and~\ref{Boundedness of pS}, for any $\Omega^\prime$ compactly contained in $\Omega$ and $T \in (0,\infty)$,
we can take a positive constant $\tilde{c}$ such that
\begin{equation}\label{bounded from above and below1}
0<\tilde{c} \leq u \leq M=:e^{\frac{1}{q}\int_0^T\|\nabla u(t)\|_p^p\,dt}\|u_{0}\|_{\infty} \quad \textrm{in}\quad  \Omega^\prime \times [0, T].
\end{equation}
As discussed in \cite[Section 5.2]{Kuusi-Misawa-Nakamura}, by~\eqref{bounded from above and below1}, we can rewrite the first equation of~\eqref{pS} as follows : Set $v:=u^q$, which is equivalent to $u=v^\frac{1}{q}$ and put $g:=\frac{1}{q} v^{1/q-1}$ and then, it is easily seen that the first equation of~\eqref{pS} is equivalent to
\begin{equation}\label{eq. of v}
\partial_tv-\mathrm{div} \big(|\nabla v|^{p-2}g^{p-1}\nabla v \big)=\lambda(t)v \quad \textrm{in}\quad \Omega^\prime \times [0,T]
\end{equation}
and thus, $v$ is a positive and bounded weak solution of the evolutionary $p$-Laplacian equation~\eqref{eq. of v}. By \eqref{bounded from above and below1} $g$ is uniformly elliptic and bounded in $\Omega'_T$. 
\medskip

The following H\"{o}lder continuity is proved via the local energy inequality for a local weak solution $v$ to~\eqref{eq. of v} (\cite[Lemma C.1]{Kuusi-Misawa-Nakamura}) and standard iterative real analysis methods. See also\cite[Chapter III]{DiBenedetto1} or \cite[Section 4.4, pp.44--47]{Urbano} for more details.

\begin{thm}[H\"{o}lder continuity {\cite[Theorem 5.6]{Kuusi-Misawa-Nakamura}}]\label{Holder continuity}
Let $v$ be a positive and bounded weak solution to (\ref{eq. of v}). Then $v$ is locally H\"{o}lder continuous in $\Omega^\prime_T$ with a H\"older exponent $\beta \in (0, 1)$ on a parabolic metric $|x|+|t|^{1/p}$.
\end{thm}
By a positivity and boundedness as in (\ref{bounded from above and below1}) and a H\"older continuity in Theorem~\ref{Holder continuity}, we see that the coefficient $g^{p-1}$ is H\"older continuous and thus, obtain a H\"older continuity of its spatial gradient.

\begin{thm}[Gradient H\"{o}lder continuity {\cite[Theorem 5.7]{Kuusi-Misawa-Nakamura}}]\label{Gradient Holder continuity}
Let $v$ be a positive and bounded weak solution to (\ref{eq. of v}). Then, there exist a positive constant $C$ depending only on $n, p, \tilde{c}, M, \lambda(0), \beta, \|\nabla v\|_{L^p(\Omega'_T)},  [g]_{\beta, \Omega'_T}$ and $[v]_{\beta, \Omega^\prime_T}$ and a positive exponent $\alpha<1$ depending only on $n,p$ and $\beta$ such that $\nabla v$ is locally H\"{o}lder continuous in $\Omega'_T$ with an exponent $\alpha$ on the usual parabolic metric. Furthermore, its H\"{o}lder constant is bounded above by $C$, where $[f]_\beta$ denote the H\"older semi-norm of a H\"older continuous function $f$ with a H\"older exponent~$\beta$.
\end{thm}
\medskip

By using an elementary algebraic estimate and a interior positivity, boundedness
and a H\"older regularity of $v$ and its gradient $\nabla v$ in Theorems~\ref{Holder continuity} and~\ref{Gradient Holder continuity},
we also obtain a local H\"{o}lder regularity of the weak solution $u$ to~\eqref{pS} and its gradient $\nabla u$, which gives our final assertion in Theorem~\ref{mainthm}. 

\begin{thm} [H\"{o}lder regularity for the $p$-Sobolev flow {\cite[Theorem 5.7]{Kuusi-Misawa-Nakamura}}]
Let $u$ be a positive and bounded weak solution to~\eqref{pS}. Then, there exist a positive exponent $\gamma<1$ depending only on $n,p,\beta, \alpha$ and a positive constant $C$ depending only on $n, p, \tilde{c}, M, \lambda(0), \beta, \alpha,\|\nabla u\|_{L^p(\Omega'_T)},  [g]_{\beta, \Omega'_T}$ and $[v]_{\beta, \Omega'_T}$ such that, both $u$ and $\nabla u$ are locally H\"{o}lder continuous in $\Omega'_T$ with an exponent $\gamma$ on a parabolic metric $|x|+|t|^{1/p}$ and on the parabolic one, respectively. The H\"{o}lder constants are bounded above by $C$, where $[f]_\beta$ denote the H\"older semi-norm of a H\"older continuous function $f$ with a H\"{o}lder exponent $\beta$.

\end{thm}
\medskip

From Steps 1 to 3, the proof of Theorem~\ref{mainthm} is concluded.

\end{proof}

\begin{appendices}

\makeatletter
\renewcommand{\thefigure}{\Alph{section}.\arabic{figure}}
\@addtoreset{figure}{section}
\makeatother

\section{Refined Expansion of Positivity}\label{Refinement of Expansion of Positivity}
This section is devoted to the refinement of the expansion of positivity which is proven in \cite[Section 4]{Kuusi-Misawa-Nakamura}. Firstly, we give the transformation stretching the time-interval, which is needed for the proof of Theorem~\ref{ex. thm prime} and Proposition~\ref{ex. cor prime}.
\medskip

In this section, following \cite[Sections 3 and 4]{Kuusi-Misawa-Nakamura}, we consider the \emph{doubly nonlinear equations of $p$-Sobolev flow type}:
\begin{equation}\label{pST}
\begin{cases}
\,\partial_tu^q-\Delta_pu=cu^q\quad &\mathrm{in}\,\,\Omega_T\\
\, 0 \leq u \leq M \quad &\mathrm{on}\,\,\partial_{p}\Omega_T,  \tag{$p$ST}
\end{cases}
\end{equation}
where $T \in (0,\infty)$, $u=u(x,t):\Omega_{T} \to \mathbb{R}$ be a nonnegative real valued function, and $c$ and $M$ are nonnegative constant and positive one, respectively. Here the initial value $u_{0}$ is in the Sobolev space $W_{0}^{1,p}(\Omega)$, positive and bounded in $\Omega$.
\smallskip

As mentioned in \cite[Remark 3.3]{Kuusi-Misawa-Nakamura}, a nonnegative weak solution of $p$-Sobolev flow~\eqref{pS} is 
a weak supersolution of~\eqref{pST} with $c=0$.
\medskip

Here we recall the fundamental positivity results, proved in \cite[Section 4]{Kuusi-Misawa-Nakamura}, which are referred later. 
\begin{prop}[{\cite[Proposition 4.1]{Kuusi-Misawa-Nakamura}}]
\label{ex. prop}Let $u$ be a nonnegative weak supersolution of~\eqref{pST}. Let $B_\rho (x_0) \subset \Omega$ with center $x_0 \in \Omega$ and radius $\rho > 0$, and $t_0 \in (0, T]$. Suppose that
\begin{equation}\label{ex. propassumption}
\big|B_\rho(x_0) \cap \{u(t_0) \geq L\} \big| \geq \alpha |B_\rho|
\end{equation}
holds for some $L>0$ and $\alpha \in (0,1]$. Then there exist positive numbers $\delta_{0},\,\varepsilon_{0} \in (0,1)$ depending only on $p,n$ and $\alpha$ and independent of $L$ such that
\begin{equation}\label{ex. propconc}
\big|B_\rho(x_0) \cap \{u(t) \geq \varepsilon L\}\big| \geq \frac{\alpha}{2} |B_\rho|
\end{equation}
holds for any positive $\delta \leq \delta_0$, any positive $\varepsilon \leq \varepsilon_0$ and all $t \in [t_0,\,t_0+\delta L^{q+1-p}\rho^p ]$. Here, if $t_0$ is very close to $T$, then $\delta> 0$ is chosen so small that $\delta L^{q+1-p} \rho^p = T-t_0$.
\end{prop}
\begin{proof}
This proposition is proved by combination of the De Giorgi iteration method and the following Caccioppoli type estimate. See \cite[Proposition 4.1]{Kuusi-Misawa-Nakamura} for detailed proof.
\end{proof}

\begin{prop}[{\cite[Proposition 3.8]{Kuusi-Misawa-Nakamura}}]
Let $k\geq 0$. Let $u$ be a nonnegative weak supersolution of~\eqref{pST}.  Let $K$ be a subset compactly contained in $\Omega$, and $0< t_1 < t_2 \leq T$. Here we use the notation $K_{t_1, t_2} = K \times (t_1, t_2)$. Let $\zeta$ be a Lipschitz function such that $\zeta=0$  outside $K_{t_1, t_2}$. Then, there exists a positive constant $C$ depending only on $p,n$ such that
\begin{align}\label{local energy ineq}
\lefteqn{
\esssup_{t_{1}<t<t_{2}}\int_{K \times \{t\}}(k-u)_{+}^{q+1}\zeta^{p}\,dx+\int_{K_{t_1,t_2}}|\nabla (k-u)_{+}\zeta |^{p}dxdt
} \qquad &
\notag \\
&\leq C\int_{K \times \{t_{1}\}}k^{q-1}(k-u)_{+}^{2}\zeta^{p}\,dx+C\int_{K_{t_1,t_2}}(k-u)_{+}^{p}|\nabla \zeta|^{p}\,dxdt \notag \\
&\qquad +C\int_{K_{t_1,t_2}}k^{q-1}(k-u)_{+}^{2}|\zeta_{t}|\,dxdt.
\end{align}
\end{prop}

We further recall the following crucial lemma.

\begin{lem}[{\cite[Lemma 4.2]{Kuusi-Misawa-Nakamura}}]\label{ex. crucial lemma}
Let $u$ be a nonnegative weak supersolution of~\eqref{pST}. Suppose further~\eqref{ex. propassumption}.
Let $Q_{4\rho}(z_{0}):=B_{4\rho}(x_{0}) \times (t_{0}, t_{0}+\delta L^{q+1-p}\rho^p) \subset \Omega_{T}$, where $\delta$ is selected in Proposition~\ref{ex. prop}. Then for any $\nu \in (0,1)$ there exists a positive number $\varepsilon_\nu$ depending only on $p,n,\alpha,\delta,\nu$ such that
\begin{equation*}
\big|Q_{4\rho}(z_{0}) \cap \{u<\varepsilon_\nu L\} \big| <\nu \big|Q_{4\rho}\big|.
\end{equation*}
\end{lem}
\begin{proof}
This lemma is also shown by the above Caccioppoli type estimate and De Giorgi's inequality. See \cite[Lemma 4.2]{Kuusi-Misawa-Nakamura} for detailed proof.
\end{proof}

\medskip

As a corollary of \cite[Theorem 4.4]{Kuusi-Misawa-Nakamura}, if a solution is positive at some time $t_0$, its positivity expands in space-time  without ''waiting time''. This corollary is used in the proof of Proposition~\ref{ex. cor prime}. See Appendix~\ref{two proofs}.

\begin{prop}[{\cite[Corollary 4.6]{Kuusi-Misawa-Nakamura}}]\label{key cor}
Let $u$ be a nonnegative weak supersolution of~\eqref{pST}. Assume  that $u(t_0)>0$ in $B_{4 \rho} (x_0) \subset \Omega$. Then there exist positive numbers $\eta_0$ and  $\tau_0$ such that \begin{equation*}
u \geq \eta_0 \quad \textrm{a.e.}\quad  \textrm{in}\quad  B_{2\rho}(x_0) \times (t_0, t_0+\tau_0).
\end{equation*}
\end{prop}
\begin{proof}
The proof is done by combination of letting $L=\inf_{B_{4\rho}(x_{0})} u(t_{0})$ in Proposition~\ref{ex. prop} and De Giorgi's iteration method. See \cite[Corollary 4.6]{Kuusi-Misawa-Nakamura} for detailed proof.
\end{proof}
%
\subsection{Expansion of Positivity via Transformation Stretching the Time-interval}
We study the expansion of interior positivity under a changing of variables stretching the time-interval.
We choose $\rho>0$ such that 
\[
Q_{16\rho}(z_{0}):=B_{16\rho}(x_{0}) \times (t_{0},t_{0}+\delta L^{q+1-p}\rho^{p} ) \subset \Omega_{T},
\]
where the a positive $\delta \leq \delta_0$ is selected in Proposition~\ref{ex. prop}. By translation, we may assume $z_{0}=(x_{0},t_{0})=(0,0)$. Following \cite[Section 5.1, pp.73--78]{DiBenedetto2}, we consider the following changing of variables stretching the time-interval:
\begin{equation}\label{def. of transformation}
y=\frac{x}{\rho}, \quad -e^{-\tau}=\frac{t-\delta L^{q+1-p}\rho^{p}}{\delta L^{q+1-p}\rho^{p}}.
\end{equation}
\begin{figure}[h]
\begin{center}
\begin{tikzpicture}[domain=0:5, samples=100, very thick] 
\draw (0,0) node[below left]{\small O}; 
\draw[thick, ->] (-0.5,0)--(5.2,0) node[right] {\small $\tau$}; 
\draw[thick, ->] (0,-0.6)--(0,2.6) node[above] {\small $t$}; 
\draw[dotted, thick] (0,2.15)--(5.2,2.15);
\draw (0,2.15) node[left] {\small $\delta L^{q+1-p}\rho^{p}$};
\draw [domain=0:4.5] plot(\x, {-2.0*exp(-\x)+2}); 
\node at(4,1) {$-e^{-\tau}=\frac{t-\delta L^{q+1-p}\rho^{p}}{\delta L^{q+1-p}\rho^{p}}$};
\end{tikzpicture}
\end{center}
\caption{Stretching the time-interval}
\end{figure}
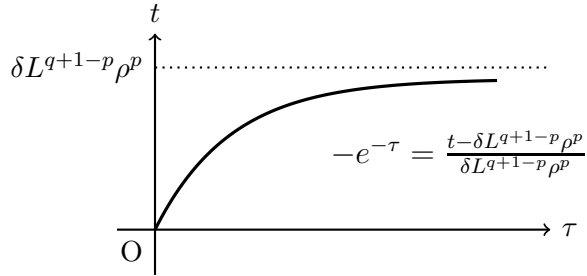

Note that this transformation $(x,t) \mapsto (y,\tau)$ maps $Q_{16\rho}(0)$ to $B_{16}\times (0,\infty)$. For a nonnegative weak solution $u$ to~\eqref{pST} we set 
\begin{equation}\label{def. of stretching v}
v(y,\tau):=\frac{1}{L}u(x,t)e^{\frac{\tau}{q+1-p}}
\end{equation}
and thus, by simple calculation, we find that $v$ satisfies the following equation in the weak sense:
\begin{equation}\label{eq1. of stretching v}
\partial_{\tau}v^{q}=\delta \Delta_{p}v+ \left(c\delta L^{q+1-p}\rho^{p}e^{-\tau}+\frac{q}{q+1-p} \right)v^{q}.
\end{equation}
Using~\eqref{def. of stretching v}, we write the conclusion~\eqref{ex. propconc} in Proposition~\ref{ex. prop} as 
\begin{equation}\label{eq2. of stretching v}
\big|B_{1} \cap \{ v(\tau) \geq \varepsilon e^{\frac{\tau}{q+1-p}}\}
 \big| \geq \frac{1}{2}\alpha |B_{1}|
\end{equation}
for every $\tau \in (0,\infty)$. Letting 
\[
k_{0}:=\varepsilon e^{\frac{\tau_{0}}{q+1-p}}, \qquad k_{j}:=\frac{k_{0}}{2^{j}}, \quad j=0,1,\ldots
\]
with the parameter $\tau_{0}>0$ determined later, 
%
we obtain from~\eqref{eq2. of stretching v} that 
\begin{equation}\label{eq3. of stretching v}
\big|B_{8} \cap \{v(\tau_{0}) \geq k_{j}\} \big| \geq \frac{\alpha}{2}8^{-n}|B_{8}|.
\end{equation}
We define the following two space-time cylinders:
\begin{align*}
Q&:=B_{8} \times (\tau_{0}+k_{0}^{q+1-p}, \tau_{0}+2k_{0}^{q+1-p}), \notag\\
Q^{\prime}&:=B_{16} \times (\tau_{0}, \tau_{0}+2k_{0}^{q+1-p}).
\end{align*}
From~\eqref{eq1. of stretching v} and similar calculation as \cite[Proposition 3.8]{Kuusi-Misawa-Nakamura}, we obtain the \emph{Caccioppoli type inequality} of $v$
\begin{align}\label{eq4. of stretching v}
\lefteqn{
\esssup_{\tau_{1}<\tau<\tau_{2}}\int_{K \times \{\tau\}}(k-v)_{+}^{q+1}\zeta^{p}\,dy+\int_{K_{\tau_1,\tau_2}}|\nabla (k-v)_{+}\zeta |^{p}dyd\tau
} \qquad &
\notag \\
&\leq \frac{C}{\delta}\int_{K \times \{\tau_{1}\}}k^{q-1}(k-v)_{+}^{2}\zeta^{p}\,dx+C\int_{K_{\tau_1,\tau_2}}(k-v)_{+}^{p}|\nabla \zeta|^{p}\,dyd\tau \notag \\
&\qquad +\frac{C}{\delta}\int_{K_{\tau_1,\tau_2}}k^{q-1}(k-v)_{+}^{2}|\zeta_{\tau}|\,dyd\tau,
\end{align}
where $k\geq 0$, $K_{\tau_1, t_2} = K \times (\tau_1, \tau_2)$ for a compact set $K \subset B_{16}$ and $\tau_{2} > \tau_{1} > 0$, and $\zeta$ is a smooth function such that $0 \leq \zeta \leq 1$ and $\zeta = 0$ outside $K_{\tau_1, \tau_2}$.
By the Caccioppoli type inequality~\eqref{eq4. of stretching v} and the very similar argument as the proof of \cite[Lemma 4.2]{Kuusi-Misawa-Nakamura}, for every $\nu>0$, there exists a natural number $J \geq \left(\frac{C(n,p)}{\nu \alpha \delta^{\frac{1}{p}}} \right)^{\frac{p}{p-1}}$ such that %
\begin{equation}\label{eq5. of stretching v}
\big|Q \cap \{v<k_{J}\} \big| <\nu |Q|.
\end{equation}
As mentioned in \cite[Remark 4.3]{Kuusi-Misawa-Nakamura} 
, we can choose $\varepsilon$ such that
\begin{equation}\label{choice of epsilon stretching v}
\varepsilon=\left(\frac{\delta}{2^{I_{1}}}\right)^{\frac{1}{q+1-p}}
\end{equation}
 for some large natural number $I_{1}$ depending only on $n,p$ and $\alpha$. We also choose $k_j$ as 
\begin{equation}\label{choice of k_j stretching v}
 k_j=\displaystyle \left(\frac{\delta e^{\tau_{0}}} {2^{I_{1}+j}} \right)^{\frac{1}{q+1-p}}\quad \textrm{for} \quad j=0,1,\ldots,J.
\end{equation}
Under such choice as above we note that $k_{0}^{q+1-p}/k_{J}^{q+1-p}=2^{J}$ is a positive integer. We divide $Q$ along time direction into parabolic cylinders of number $2^{J}$ with each time-length $k_{J}^{q+1-p}$, and set
\begin{equation*}
Q^{(\ell)}:=B_{8}\times \left(\tau_{0}+k_{0}^{q+1-p}+\ell k_{J}^{q+1-p},\,\,\tau_{0}+k_{0}^{q+1-p}+(\ell+1) k_{J}^{q+1-p}\right)
\end{equation*}
for $\ell=0,1,\ldots ,2^{J}-1$.
By~\eqref{eq5. of stretching v} there is a $Q^{(\ell)}$ such that
\begin{equation}\label{eq6. of stretching v}
\big|Q^{(\ell)} \cap \{v<k_{J}\} \big| <\nu |Q^{(\ell)}|.
\end{equation}
\begin{figure}[h]
\begin{center}
\begin{tikzpicture}[scale=1.6]
\draw[thin] (-1.0,0) -- (1.0,0);
\draw[thin] (-1.0,0.5) -- (1.0,0.5);
\draw (-1,-0.5) rectangle (1.0,1.0);
\draw[thick, densely dotted] (-1.5,1.0) -- (-1.0,1.0);
\draw[densely dotted] (-1.5,0.5) -- (-1.0,0.5);
\draw[densely dotted] (-1.5,0) -- (-1.0,0);
\draw[thick, densely dotted] (-1.5,-0.5) -- (-1.0,-0.5);
\draw[->] (-1.5,-1.0) -- (1.5, -1.0)node[right]{\footnotesize $y \in \mathbb{R}^{n}$};
\draw[->] (-1.50,-1.4)-- (-1.5,1.5)node[right]{\footnotesize $\tau$};
\draw[thick, <->] (1.5,0)-- (1.5,0.5);
\draw (2.8, 0.2)node{\footnotesize height: $k_{J}^{q+1-p}$};
\filldraw [pattern=north east lines, pattern color=blue,] (-0.99,0.01) rectangle (0.99,0.49);
\draw  (0.4,0.2)node[left]{\footnotesize $Q^{(\ell)}$};
\draw  (-1.5,1.0)node[left]{\footnotesize $\tau_{0}+2k_{0}^{q+1-p}$};
\draw  (-1.5,0.5)node[left]{\footnotesize $\tau_{0}+k_{0}^{q+1-p}+(\ell+1) k_{J}^{q+1-p}$};
\draw  (-1.5,0)node[left]{\footnotesize $\tau_{0}+k_{0}^{q+1-p}+\ell k_{J}^{q+1-p}$};
\draw  (-1.5,-0.5)node[left]{\footnotesize $\tau_{0}+k_{0}^{q+1-p}$};
\draw  (0,-0.5)node[below]{\footnotesize $B_{8}$};
\end{tikzpicture}
\end{center}
\caption{Image of $Q^{(\ell)}$}
\end{figure}
As a result, we obtain the expansion of interior positivity of time-stretched solution $v$ as follows.

\begin{prop}\label{expansion of positivity of stretching v} \normalfont 
Let $v$ be a nonnegative weak solution of~\eqref{eq1. of stretching v} in $B_{16} \times (0,\infty)$ defined by~\eqref{def. of stretching v}. Then 
\[
v \geq \frac{k_{J}}{2} \quad \textrm{a.e.}\,\,\textrm{in}\,\,B_{4}\times \left(\tau_{0}+k_{0}^{q+1-p}+\ell k_{J}^{q+1-p},\,\,\tau_{0}+k_{0}^{q+1-p}+(\ell+1) k_{J}^{q+1-p}\right).
\]
\end{prop}
\begin{proof}
The proof is done by the same argument as \cite[Theorem 4.4]{Kuusi-Misawa-Nakamura}.
\end{proof}
Under preliminaries as above, we are in position to state the main theorem in this subsection.
\begin{thm}[Expansion of interior positivity up to end time]\label{Expansion of positivity until end time} 
Let $u$ be a nonnegative weak solution to~\eqref{pST}. Let $B_{16 \rho} (x_0) \times (t_0, t_0 + \delta L^{q + 1 - p} \rho^p) \subset \Omega_T$ with $x_0 \in \Omega$, $\rho > 0$, and $t_0 \in (0, T]$. Suppose~\eqref{ex. propassumption}. Then there exist positive numbers $\eta<1$ and $\sigma<1$ depending only on $n,p, \alpha$ and independent of $L$ such that 
\begin{equation}
u \geq \eta L \quad \textrm{a.e.}\,\,\textrm{in}\,\,B_{2\rho}(x_{0})\times \left(t_{0}+(1-\sigma) \delta L^{q+1-p}\rho^{p}, t_{0}+\delta L^{q+1-p}\rho^{p} \right),
\end{equation}
where the $\delta=\delta(n,p,\alpha)>0$ is selected in Proposition~\ref{ex. prop}.
\end{thm} 
\begin{proof}
The assertion is verified by Proposition~\ref{expansion of positivity of stretching v}, the scaling back $(x,t) \leftrightarrow (y,\tau)$ and De Giorgi's iteration method as in \cite[Theorem 4.4]{Kuusi-Misawa-Nakamura}. Again, by translation, we may consider $(x_{0},t_{0})=(0,0)$. It follows from Proposition~\ref{expansion of positivity of stretching v} and~\eqref{choice of k_j stretching v} that for almost every time $\tau_{1}$, $\tau_{0}+k_{0}^{q+1-p}+\ell k_{J}^{q+1-p}<\tau_{1}<\tau_{0}+k_{0}^{q+1-p}+(\ell+1) k_{J}^{q+1-p}$
\[
v(y,\tau_{1}) \geq \frac{1}{2}\left(\frac{\delta e^{\tau_{0}}}{2^{I_{1}+J}} \right)^{\frac{1}{q+1-p}} \quad \textrm{a.e.}\,\,\textrm{in}\,\,B_{4},
\]
which,  letting $t_1$ by~\eqref{def. of transformation} with $\tau_{1}$ as
\begin{equation}\label{t1}
-e^{-\tau_{1}}=\frac{t_{1}-\delta L^{q+1-p}\rho^{p}}{\delta L^{q+1-p}\rho^{p}},
\end{equation}
%
leads to 
\begin{equation}\label{eq.1 of pf of expansion of positivity until end time}
u(x,t_{1}) \geq \frac{1}{2} \left(\frac{\delta e^{-(\tau_{1}-\tau_{0})}}{2^{I_{1}+J}} \right)^{\frac{1}{q+1-p}}L \quad \textrm{a.e.}\,\,\textrm{in}\,\,B_{4\rho}.
\end{equation}
For brevity, we set 
\begin{equation}\label{L0}
L_{0}:=\frac{1}{2} \left(\frac{\delta e^{-(\tau_{1}-\tau_{0})}}{2^{I_{1}+J}} \right)^{\frac{1}{q+1-p}}L.
\end{equation}
Since it follows from~\eqref{eq.1 of pf of expansion of positivity until end time} that $\big|B_{4\rho} \cap \{u(t_{1}) \geq L_{0}\big|=|B_{4\rho}|$, by Proposition~\ref{ex. prop}, there exist positive numbers $\tilde{\delta}<1$ and $\tilde{\varepsilon}<1$ depending only on $n$ and $p$ and independent of $L_{0}$ such that 
\begin{equation}\label{eq.2 of pf of expansion of positivity until end time}
\big|B_{4\rho} \cap \{u(t) \geq \tilde{\varepsilon}L_{0}\}\big| \geq \frac{1}{2}\big|B_{4\rho}\big|
\end{equation}
holds for every $t \in [t_{1}, t_{1}+\tilde{\delta}L_{0}^{q+1-p}\rho^{p}]$. For a positive $\theta \leq \tilde{\delta}L_{0}^{q+1-p}$ let $Q^{\theta}_{4\rho}:=B_{4\rho} \times (t_{1}, t_{1}+\theta \rho^{p})$. By Lemma~\ref{ex. crucial lemma}, for every $\tilde{\nu} \in (0,1)$ there exists a positive $\tilde{\varepsilon}_{\tilde{\nu}}$ such that 
\begin{equation}\label{eq.3 of pf of expansion of positivity until end time}
\big|Q^{\theta}_{4\rho} \cap \{u<\tilde{\varepsilon}_{\tilde{\nu}}L_{0} \big| <\tilde{\nu}|Q^{\theta}_{4\rho} \big|.
\end{equation}
As in the proof of \cite[Theorem 4.4]{Kuusi-Misawa-Nakamura}, let
\begin{align*}
&\rho_{m}:=\left(2+\frac{1}{2^{m-1}} \right)\rho, \quad \quad Q_{m}:=B_{\rho_{m}}\times (t_{1}, t_{1}+\theta \rho^{p}), \notag\\ &\theta:=\tilde{\delta} L_{0}^{q+1-p}, \quad \quad \kappa_{m}:=\left(\frac{1}{2}+\frac{1}{2^{m+1}} \right)\tilde{k}_{I_{2}}.
\end{align*}
%
%
where $\displaystyle \tilde{k}_{I_{2}}:= \frac{\tilde{\varepsilon}_{\tilde{\nu}}L_{0}}{2^{\frac{I_{2}}{q+1-p}}}$ with $\tilde{\varepsilon}_{\tilde{\nu}}$ in~\eqref{eq.3 of pf of expansion of positivity until end time} and a natural number $I_{2}$ satisfying $I_{2} \geq \left(\frac{C(n,p)}{\tilde{\nu}\tilde{\delta}^{\frac{1}{p}}} \right)^{\frac{p}{p-1}}$. 
%
%
%
It then plainly holds that
\begin{equation*}
\begin{cases}
4\rho =\rho_{0}\geq \rho_m \searrow \rho_{\infty}=2\rho \\[2mm]
Q^{\theta}_{4\rho}=Q_0
\supset Q_m \searrow Q_{\infty}=B_{2\rho} \times (t_{1}, t_{1}+\theta\rho^{p}) \\[2mm]
\tilde{k}_{I_{2}}=\kappa_{0} \geq \kappa_m \searrow \kappa_{\infty}=\tilde{k}_{I_{2}}/2.
\end{cases}
\end{equation*}
Letting $\displaystyle Y_m:=\dashint_{Q_m} \chi_{\{(\kappa_m-u)_+>0\}}\,dz$ and using the Caccioppoli type inequality~\eqref{local energy ineq} with the cut-off function $\zeta=\zeta (x)$ in $Q_{m}$  as the proof of \cite[Corollary 4.6]{Kuusi-Misawa-Nakamura}, we obtain
\begin{equation*}
 Y_{m+1} \leq Cb^m Y_m^{1+\frac{p}{n}},\quad m=0,1,\ldots,
 \end{equation*}
 where $b:=2^{p(1+\frac{p}{n})+p\frac{n+q+1}{n}}>1$.
 By the fast geometric convergence lemma (see \cite[Lemma 2.3]{Kuusi-Misawa-Nakamura} and also \cite[Lemma 4.1, p12]{DiBenedetto1}), if
 \begin{equation}\label{condition2 on Y_{0} end time}
 Y_0 \leq C^{-\frac{n}{p}}b^{-(\frac{n}{p})^2}=:\tilde{\nu}_0,
 \end{equation}
 then
 \begin{equation}\label{limit2 Y_{n} end time}
 Y_m \to 0\quad \textrm{as}\quad  m \to \infty.
 \end{equation}
 Eq.~\eqref{condition2 on Y_{0} end time} follows from~\eqref{eq.3 of pf of expansion of positivity until end time}
 with $\tilde{\nu}=\tilde{\nu_0}$ and thus,~\eqref{limit2 Y_{n} end time} gives that
\[
u \geq \frac{\tilde{k}_{I_{2}}}{2} \quad \,\,\textrm{in}\,\,B_{2\rho} \times (t_{1}, t_{1}+\theta \rho^{p}),
\]
that is, 
\begin{equation}\label{eq.4 of pf of expansion of positivity until end time}
u \geq \frac{\tilde{\varepsilon}}{4} \left(\frac{\delta e^{-(\tau_{1}-\tau_{0})}}{2^{I+J}} \right)^{\frac{1}{q+1-p}}L\quad \,\,\textrm{in}\,\,B_{2\rho}\times (t_{1},t_{1}+\tilde{\delta}L_{0}^{q+1-p}\rho^{p}),
\end{equation}
where we put $I:=I_{1}+I_{2} \in \mathbb{N}$, which depends only on $n,p,\alpha, \varepsilon$ and $\delta$ and independent of $L$ and $\tilde{\varepsilon}=\tilde{\varepsilon}_{\tilde{\nu}_{0}}$. Here we note that 
\[
{\tilde \delta} L_{0}^{q+1-p}= {\tilde \delta}\frac{\delta e^{- (\tau_{1}-\tau_{0})}}{2^{I_{1}+J}}\left(\frac{L}{2}\right)^{q+1-p}.
\]
%
%
%
 %
Thus, the inequality~\eqref{eq.4 of pf of expansion of positivity until end time} is written as 
\begin{equation}\label{eq.5 of pf of expansion of positivity until end time}
u(x,t) \geq \frac{\tilde{\varepsilon}}{4} \left(\frac{\delta e^{-(\tau_{1}-\tau_{0})}}{2^{I+J}} \right)^{\frac{1}{q+1-p}}L\quad \,\,\textrm{in}\,\,B_{2\rho}
\end{equation}
for all times $t_{1} <t< t_{1}+\tilde \delta\frac{\delta e^{-(\tau_{1}-\tau_{0})}}{2^{I_{1} + J}}
\left(\frac{L}{2}\right)^{q+1-p}\rho^p$.
The transformed $\tau_{0}$ is still free of choice, and it will be selected as follows. By the change of variables~\eqref{def. of transformation} with~\eqref{t1}, $\tau_{0}$ is chosen as
\begin{align}\label{eq.6 of pf of expansion of positivity until end time}
&\delta L^{q+1-p}\rho^{p}-t_{1}=\tilde \delta \frac{\delta e^{-(\tau_{1}-\tau_{0})}}{2^{I_{1}+J}}\left(\frac{L}{2}\right)^{q+1-p}\rho^p \notag\\[1mm]
&\!\!\!\!\iff \tau_{0}=\log \left(\frac{2^{I_{1}+J+q+1-p}}{\tilde{\delta}} \right).
\end{align}
This $\tau_{0}$ depends only on $n,p, \varepsilon, \delta$ and $\alpha$ because $\tilde{\delta}$ and $\tilde{\varepsilon}$ depend only on $n, p$. Therefore,~\eqref{eq.5 of pf of expansion of positivity until end time} holds for all times 
\begin{equation}\label{eq.7 of pf of expansion of positivity until end time}
t_{1}=\delta L^{q+1-p}\rho^{p}-\tilde \delta \frac{\delta e^{-(\tau_{1}-\tau_{0})}}{2^{I_{1}+J}}\left(\frac{L}{2}\right)^{q+1-p}\rho^p<t \leq \delta L^{q+1-p}\rho^{p}.
\end{equation}
Lastly, we will estimate the above ''left edge time'' $t_{1}$. Since 
\begin{align}\label{eq.8 of pf of expansion of positivity until end time}
\tau_{1}\leq \tau_{0}+2k_{0}^{q+1-p} \iff e^{-\tau_{1}} \geq e^{-(\tau_{0}+2k_{0}^{q+1-p})}, \notag\\
k_{0}^{q+1-p}=\varepsilon^{q+1-p}e^{\tau_{0}}\leq e^{\tau_{0}} \iff e^{-2k_{0}^{q+1-p}}\geq e^{-2e^{\tau_{0}}},
\end{align}
it follows from~\eqref{eq.6 of pf of expansion of positivity until end time} and~\eqref{eq.8 of pf of expansion of positivity until end time} that 
\begin{align*}
t_{1}&=\delta L^{q+1-p}\rho^{p}-\tilde \delta \frac{\delta e^{-(\tau_{1}-\tau_{0})}}{2^{I_{1}+J}}\left(\frac{L}{2}\right)^{q+1-p}\rho^p\notag\\
&=(1-e^{-\tau_{1}})\delta L^{q+1-p}\rho^{p}\notag\\[1mm]
&\leq \left(1-e^{-(\tau_{0}+2k_{0}^{q+1-p})}\right)\delta L^{q+1-p}\rho^{p} \notag\\
&\leq \left(1-e^{-(\tau_{0}+2e^{\tau_{0}})}\right)\delta L^{q+1-p}\rho^{p}. 
\end{align*}
Here we set $\sigma:=e^{-(\tau_{0}+2e^{\tau_{0}})} \in (0,1)$. This together with~\eqref{eq.5 of pf of expansion of positivity until end time}
implies that
\[
u \geq \frac{\tilde{\varepsilon}}{4} \left(\frac{\delta e^{-2\tau_{0}}}{2^{I+J}} \right)^{\frac{1}{q+1-p}}L\quad \,\,\textrm{in}\,\,B_{2\rho}\times ((1-\sigma)\delta L^{q+1-p},\,\delta L^{q+1-p}\rho^{p})
\]
and thus, letting $\eta:=\frac{\tilde{\varepsilon}}{4} \left(\frac{\delta e^{-2\tau_{0}}}{2^{I+J}} \right)^{\frac{1}{q+1-p}}$, we complete the proof.
\end{proof}


\subsection{Proof of Theorem~\ref{ex. thm prime} and Proposition~\ref{ex. cor prime}}\label{two proofs}
This subsection is devoted to the detailed proof of Theorem~\ref{ex. thm prime} and Proposition~\ref{ex. cor prime}.
\medskip

We shall prove Theorem~\ref{ex. thm prime} by using Theorem~\ref{Expansion of positivity until end time} and a method of chain of finitely many balls as used in Harnack's inequality for harmonic functions, which is so-called \textit{Harnack chain} (see \cite[Theorem 11, pp.32--33]{Evans} and~\cite{AKN,KMN} in the $p$-parabolic setting). Here we use the special choice of parameters, as explained in Theorem~\ref{Expansion of positivity until end time} (see also \cite[Theorem 4.4]{Kuusi-Misawa-Nakamura}).

\bigskip

\noindent

\begin{proof}[\normalfont \textbf{Proof of Theorem~\ref{ex. thm prime}}]
We follow a similar argument as \cite[Section 4.3]{Kuusi-Misawa-Nakamura}.
\medskip

We will prove the assertion in four steps.

\medskip

\textbf{Step 1}: \,Since $\overline{\Omega^\prime}$ is compact, it is covered by finitely many balls $\{B_{\rho}(x_j)\}_{j=1}^N\,\,(x_j \in \Omega^\prime,\,j=1,2,\ldots,N)$ with $N=N(\Omega^\prime, \rho)$, such that
\begin{equation*}
\Omega^\prime \subset \bigcup_{j=1}^N B_{\rho}(x_j), \quad  \rho <|x_i- x_{i+1}|<2\rho,\, \,\,B_{16\rho}(x_{i}) \subset \Omega,\,\,\textrm{for all}\,\,  1 \leq i \leq N,
\end{equation*}
where we put $x_{N+1}=x_1$.
For brevity we denote $B_{\rho}(x_j)$ by $B_j$ for each $j=1,2,\ldots, N$.
\bigskip

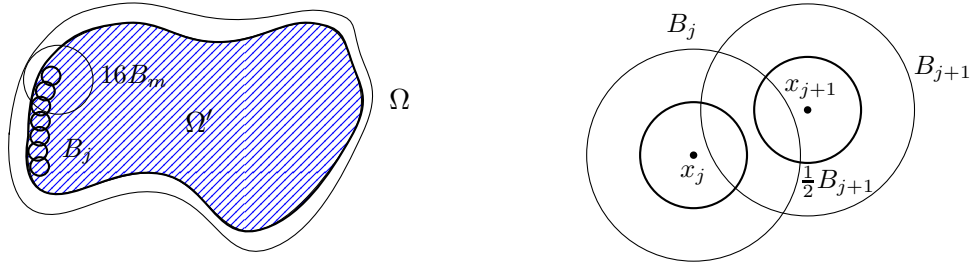
\begin{figure}[h]
\begin{center}
\begin{tikzpicture}[scale=0.50]
            \draw  plot[draw=black, smooth, tension=.8] coordinates {(-4.0,0.55) (-3.3,2.8) (-1.2,4.0) (1.8,3.6) (4.4,3.8) (5.2,2.8) (5.3,0.6) (2.8,-2.5) (0,-1.0) (-3.3,-1.5) (-4.0,0.6)};
            \filldraw[pattern=north east lines, pattern color=blue, thick]  plot[smooth, tension=.8] coordinates {(-3.5,0.5) (-3,2.5) (-1,3.5) (1.5,3) (4,3.5) (5,2.5) (5,0.5) (2.5,-2) (0,-0.5) (-3,-1.0) (-3.5,0.5)};
            \draw plot[smooth, tension=.7] coordinates {(-3.5,0.5) (-3,2.5) (-1,3.5) (1.5,3) (4,3.5) (5,2.5) (5,0.5) (2.5,-2) (0,-0.5) (-3,-1.0) (-3.5,0.5)};
           \draw (1,1.5)node[below]{\small $\Omega'$};
           \draw (5.7,1.5)node[right]{\small $\Omega$};
           \draw (-1.9,2.1)node[right]{\footnotesize $16B_{m}$};
           \draw[thin] (-2.71,2.0) circle (26pt);
            \draw (-1.5,0.1)node[left]{\footnotesize $B_{j}$};
            \draw[thick] (-3.2,-0.3) circle (7pt);
            \draw[thick] (-3.25,0.1) circle (7pt);
            \draw[thick] (-3.2,0.5) circle (7pt);
            \draw[thick] (-3.18,0.9) circle (7pt);
            \draw[thick] (-3.17,1.3) circle (7pt);
             \draw[thick] (-3.05,1.7) circle (7pt);
             \draw[thick] (-2.9,2.1) circle (7pt);
             \draw[thin] (14,0) circle (80pt);
            \draw[thick] (14,0) circle (40pt);
            \draw[thick] (17,1.2) circle (40pt);
            \draw[thin] (17,1.2) circle (80pt);
            \draw (13,3.4)node[right]{\footnotesize $B_j$};
            \draw (16.5,-0.7)node[right]{\footnotesize $\frac{1}{2}B_{j+1}$};
            \draw (19.5,2.3)node[right]{\footnotesize $B_{j+1}$};
            \draw (14,0)node[below]{\footnotesize $x_j$};
            \draw (17.1,1.2)node[above]{\footnotesize $x_{j+1}$};
            \filldraw[thick] (14,0) circle (2pt);
            \filldraw[thick] (17,1.2) circle (2pt);
 \end{tikzpicture}
\end{center}
\caption{Harnack's chain argument}
\end{figure}
By~\eqref{ex. propassumption prime}, there exists at least one $B_j=B_\rho (x_j)$, denoted by  $x_1=x_j$ and $B_1=B_j$, such that
\begin{equation*}
|B_1 \cap \{u(t_0) \geq L\}| \geq \frac{\alpha}{2^n} |B_1|.
\end{equation*}
Thus, by Theorem~\ref{Expansion of  positivity until end time}, there exists positive numbers $\delta_0, \varepsilon_0, \sigma_{0}, \eta_{1} \in (0,1)$ depending only on $p, n$ and $\alpha_{0}=\alpha$ and independent of $L$ such that
\begin{equation}\label{eq.1 of revised ex. prime}
u \geq \eta_{1}L\quad \textrm{a.e.}\,\,\textrm{in} \quad B_{1}\times \big( t_{0}+(1-\sigma_{0})\delta_{0}L^{q+1-p}\rho^{p},\,t_{0}+\delta_{0}L^{q+1-p}\rho^{p}\big), 
\end{equation}
where $\sigma_0:=e^{-(\tau_0+2e^{\tau_0})}$, $e^{\tau_0} = C(n,p)2^{I_0+J_0}$ for some $I_{0},\,J_{0} \in \mathbb{N}$ depending only on $n,p$ and $\alpha_{0}$, and $\displaystyle \eta_{1}=\frac{\tilde{\varepsilon}_{0}}{4} \left(\frac{\delta_{0}e^{-2\tau_{0}}}{2^{I_{0}+J_{0}}}\right)^{\frac{1}{q+1-p}}$ for some $\tilde{\varepsilon}_{0}=\tilde{\varepsilon}_{0}(n,p) \in (0,1)$. We put $\eta_0=1$ for later reference.
\medskip

\textbf{Step 2}: \,  By $\rho<|x_1-x_2|<2\rho$,
\begin{equation*}
D_1:=B_1 \cap B_2 \neq \varnothing.
\end{equation*}
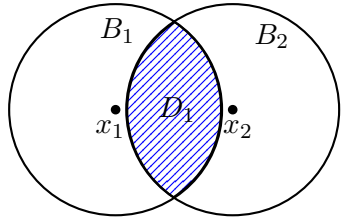
\begin{figure}[h]
\begin{center}
\begin{tikzpicture}[scale=0.7]
\filldraw[pattern=north east lines, pattern color=blue,thick] (0.9,-1.65) arc(-56:56:2);
\filldraw[pattern=north east lines, pattern color=blue,thick] (0.9,1.65) arc(124:236:2);
\draw[thick](-0.2,0) circle(2) (2,0) circle (2);
 \draw (0.4,0)node[right]{\small $D_1$};
\draw (-0.7,1.5)node[right]{\small $B_1$};
  \draw (2.2,1.4)node[right]{\small $B_2$};
   \draw (-0.3,0)node[below]{\small $x_1$};
  \draw (2.1,0)node[below]{\small $x_2$};
    \filldraw[thick] (-0.2,0) circle (2pt);
  \filldraw[thick] (2,0) circle (2pt);
\end{tikzpicture}
\end{center}
\caption{Intersection of two balls}
\end{figure}
\bigskip

Via~\eqref{eq.1 of revised ex. prime}, we have
\begin{equation}\label{eq.2 of revised ex. prime}
u \geq \eta_1L \quad \textrm{a.e.}\,\,D_1 \times \mathcal{I}_0,
\end{equation}
where let $\mathcal{I}_0:=\big(t_{0}+(1-\sigma_{0})\delta_{0}L^{q+1-p}\rho^{p},\,t_{0}+\delta_{0}L^{q+1-p}\rho^{p}\big)$. By~\eqref{eq.2 of revised ex. prime}, for any $t_1 \in \mathcal{I}_0$, 
\begin{equation*}
|D_1 \cap \{u(t_1) \geq \eta_1 L\}|=|D_1|,
\end{equation*}
which is, setting $\displaystyle \alpha_1:=\frac{|D_1|}{|B_2|} \in (0,1)$,
\begin{equation*}
\big|B_2 \cap \{u(t_1) \geq \eta_1 L\}\big| \geq \alpha_1 \big|B_2\big|.
\end{equation*}
By the very same argument as Step 1, there exist positive numbers $\delta_{1}, \sigma_{1},\eta_{2}\in (0,1)$ depending only on $p,\,n$ and $\alpha_1$ and independent of $L$ such that
\begin{equation}\label{eq.3 of revised ex. prime}
u \geq \eta_{2}L\quad \textrm{a.e.}\,\,\textrm{in} \quad B_{2}\times \big( t_{1}+(1-\sigma_{1})\delta_{1}(\eta_{1}L)^{q+1-p}\rho^{p},\,t_{1}+\delta_{1}(\eta_{1}L)^{q+1-p}\rho^{p}\big), 
\end{equation}
where $\sigma_1:=e^{-(\tau_1+2e^{\tau_1})}$, $e^{\tau_1} = C(n,p)2^{I_1+J_1}$ for some large $I_{1},\,J_{1} \in \mathbb{N}$ depending only on $n,p$ and $\alpha_{1}$, and $\displaystyle \eta_{2}=\frac{\tilde{\varepsilon}_{1}}{4} \left(\frac{\delta_{1}e^{-2\tau_{1}}}{2^{I_{1}+J_{1}}}\right)^{\frac{1}{q+1-p}}\eta_{1}$ for some $\tilde{\varepsilon}_{1}=\tilde{\varepsilon}_{1}(n,p) \in (0,1)$. Here we choose $t_1 \in \mathcal{I}_{0}$ as 
\begin{align*}
&t_1:=t_0+(\delta_{0}-\delta_{1}\eta_1^{q+1-p})L^{q+1-p}\rho^p > 0 \notag\\[2mm]
&\!\!\!\!\iff t_{1}+\delta_{1}(\eta_1L)^{q+1-p}\rho^{p}=t_{0}+\delta_{0}L^{q+1-p}\rho^{p}.
\end{align*}
Note that this choice of $t_{1} \in \mathcal{I}_{0}$ is admissible by
\begin{align*}
&t_{1}-\left(t_{0}+(1-\sigma_{0})\delta_{0}L^{q+1-p}\rho^{p}\right)\notag\\
&=\left(\sigma_0-\delta_{1}\left(\frac{\tilde{\varepsilon_{0}}}{4}\right)^{q+1-p}\frac{e^{-2\tau_{0}}}{2^{I_{0}+J_{0}}}\right)\delta_{0}L^{q+1-p}\rho^{p}> 0
\end{align*}
and $\delta_1$ can be chosen as small.
\smallskip

\textbf{Step 3}:\, We will proceed by induction on $m$. Suppose that for some $m \in \{1,2.\ldots ,N\}$
\begin{equation}\label{eq.4 of revised ex. prime}
u \geq \eta_mL \quad \textrm{a.e.}\,\,\textrm{in}\,\,B_m\times \mathcal{I}_{m-1}.
\end{equation}
Here let
\begin{equation*}
\mathcal{I}_{m-1}:=\big( t_{m-1}+(1-\sigma_{m-1})\delta_{m-1}(\eta_{m-1}L)^{q+1-p}\rho^{p},\,t_{m-1}+\delta_{m-1}(\eta_{m-1}L)^{q+1-p}\rho^{p}\big)
\end{equation*}
with $t_{m-1}:=t_{0}+(\delta_{0}-\delta_{m-1}\eta_{m-1}^{q+1-p})L^{q+1-p}\rho^{p}$,
where $\delta_{m-1}, \sigma_{m-1} \in (0,1), \tau_{m-1}>0$ are determined inductively as follows: $\sigma_{m-1}:=e^{-(\tau_{m-1}+2e^{\tau_{m-1}})}$, $e^{\tau_{m-1}} = C(n,p)2^{I_{m-1}+J_{m-1}}$ for some large $I_{m-1},\,J_{m-1} \in \mathbb{N}$ depending only on $n, p$ and $\alpha_{m-1} = \tfrac{|D_{m-1}|}{|B_{m}|}=\tfrac{|D_{1}|}{|B_{2}|}=\alpha_{1}$ as before, and $\displaystyle \eta_{m}=\frac{\tilde{\varepsilon}_{m-1}}{4} \left(\frac{\delta_{m-1}e^{-2\tau_{m-1}}}{2^{I_{m-1}+J_{m-1}}}\right)^{\frac{1}{q+1-p}}\eta_{m-1}$ for some $\tilde{\varepsilon}_{m-1}=\tilde{\varepsilon}_{m-1}(n,p) \in (0,1)$.

By $\rho <|x_m-x_{m+1}|<2\rho$ again,
\begin{equation*}
D_m:=B_m \cap B_{m+1}\neq \varnothing
\end{equation*}
and thus,~\eqref{eq.4 of revised ex. prime} yields that
\begin{equation}\label{eq.5 of revised ex. prime}
u \geq \eta_mL \quad \textrm{a.e.}\,\,\textrm{in}\,\,D_m \times \mathcal{I}_{m-1}.
\end{equation}
By~\eqref{eq.5 of revised ex. prime}, for any $t_m \in \mathcal{I}_{m-1}$,
\begin{equation*}
\big|B_{m+1} \cap \{u(t_m) \geq \eta_mL\}\big| \geq \alpha_m\big|B_{m+1}\big|,
\end{equation*}
where let $\displaystyle \alpha_m:=\frac{|D_m|}{|B_{m+1}|}=\frac{|D_{1}|}{|B_{2}|}=\alpha_{1} \in (0,1)$. Again, similarly as in Steps 1 and 2, 
there exist positive numbers $\delta_{m}, \sigma_{m},\eta_{m+1}\in (0,1)$ depending only on $p,\,n$ and $\alpha_1$ and independent of $L$ such that
\begin{equation}\label{eq.6 of revised ex. prime}
u \geq \eta_{m+1}L\quad \textrm{a.e.}\,\,\textrm{in} \quad B_{m+1}\times \big( t_{m}+(1-\sigma_{m})(\delta_{m}L)^{q+1-p}\rho^{p},\,t_{m}+(\delta_{m}L)^{q+1-p}\rho^{p}\big), 
\end{equation}
where $\sigma_m:=e^{-(\tau_{m}+2e^{\tau_{m}})}$, $e^{\tau_m} = C(n,p)2^{I_m+J_m}$ for some large $I_{m},\,J_{m} \in \mathbb{N}$ depending only on $n,p$ and $\alpha_{1}$, and $\displaystyle \eta_{m+1}=\frac{\tilde{\varepsilon}_{m}}{4} \left(\frac{\delta_{m}e^{-2\tau_{m}}}{2^{I_{m}+J_{m}}}\right)^{\frac{1}{q+1-p}}\eta_{m}$ for some $\tilde{\varepsilon}_{m}=\tilde{\varepsilon}_{m}(n,p) \in (0,1)$. Again, we choose $t_m \in \mathcal{I}_{m-1}$ as 
\begin{align*}
&t_m:=t_0+(\delta_{0}-\delta_{m}\eta_m^{q+1-p})L^{q+1-p}\rho^p > 0 \notag\\[2mm]
&\!\!\!\!\iff t_{m}+\delta_{m}(\eta_{m}L)^{q+1-p}\rho^{p}=t_{0}+\delta_{0}L^{q+1-p}\rho^{p},
\end{align*}
because
\begin{align*}
&t_{m}-\left(t_{m-1}+(1-\sigma_{m-1})\delta_{m-1}(\eta_{m-1}L)^{q+1-p}\rho^{p}\right)\notag\\
&=\left(\sigma_{m-1}-\delta_{m}\left(\frac{\tilde{\varepsilon}_{m-1}}{4}\right)^{q+1-p}\frac{e^{-2\tau_{m-1}}}{2^{I_{m-1}+J_{m-1}}}\right)\delta_{m-1}(\eta_{m-1}L)^{q+1-p}\rho^{p}> 0
\end{align*}
and $\delta_m$ can be taken to be small enough. Thus our induction on $m$ is done.
\smallskip

\textbf{Step 4}:\, By Step 3, we have, for all $m=1,2,\ldots,N$,
\begin{equation}\label{eq.7 of revised ex. prime}
u \geq \eta_{m+1}L \quad \textrm{a.e.}\,\,\textrm{in}\,\,B_{m+1}\times \mathcal{I}_{m},
\end{equation}
where let $B_{N+1}:=B_1$. Since, by construction, $\{\eta_m\}_{m=1}^{N+1}$ is decreasing, it follows from~\eqref{eq.7 of revised ex. prime} that, for all $m=1,2,\ldots,N$,
\begin{equation*}
u \geq \eta_{N+1}L \quad \textrm{a.e.}\,\,\textrm{in}\,\,B_{m+1}\times \mathcal{I}_{N},
\end{equation*}
where we set 
\[
\mathcal{I}_{N}:=\big( t_{N}+(1-\sigma_{N})(\delta_{N}L)^{q+1-p}\rho^{p},\,t_{N}+(\delta_{N}L)^{q+1-p}\rho^{p}\big)
\]
with $t_{N}=t_{0}+(\delta_{0}-\delta_{N}\eta_{N}^{q+1-p})L^{q+1-p}\rho^{p}$.
Therefore we complete the proof. 
\end{proof}
Lastly, we will prove Proposition~\ref{ex. cor prime}.

\begin{proof}[\normalfont \textbf{Proof of Proposition~\ref{ex. cor prime}}]
Since $\overline{\Omega^\prime}$ is compact, it is covered by finitely many balls $\{B_{\rho}(x_j)\}_{j=1}^N\,\,(x_j \in \Omega^\prime,\,j=1,2,\ldots,N)$, where $N=N(\Omega^\prime, \rho)$, such that
\begin{equation*}
\Omega^\prime \subset \bigcup_{j=1}^N B_{\rho}(x_j), \quad  \rho <|x_i- x_{i+1}|<2\rho,\, \,\,B_{16\rho}(x_{i}) \subset \Omega,\,\,\textrm{for all}\,\,  1 \leq i \leq N,
\end{equation*}
where we put $x_{N+1}=x_1$.
For brevity we denote $B_{\rho}(x_j)$ by $B_j$ for each $j=1,2,\ldots, N$ and let $4B_{j}:=B_{4\rho}(x_{j})$. By assumption, $u(t_{0})>0$ in each ball $4B_{j}$, $j=1,\ldots, N$.  Let $L_{1} = \inf_{4B_{1}} u (t_{0})$. By Corollary~\ref{key cor} with $L=L_{1}$ (see the proof of \cite[Corollary 4.6]{Kuusi-Misawa-Nakamura}) there exist positive numbers $\eta_{1}$ and $\tau_{1}$ depending on $n$ and $p$ such that
\begin{equation*}
u \geq \eta_{1}L_{1}\quad \textrm{a.e.}\,\,\textrm{in}\,\,B_{1}\times (t_{0},t_{0}+\tau_{1} L_{1}^{q +1-p} \rho^p).
\end{equation*}
Letting $L_{2}:=\inf_{4B_{2}} u (t_{0})$ and applying Corollary~\ref{key cor} with $L=L_{2}$, there exists positive numbers $\eta_{2}$ and $\tau_{2}$ depending on $n$ and $p$
 such that
\begin{equation*}
u \geq \eta_{2} L_{2}\quad \textrm{a.e.}\,\,\textrm{in}\,\,B_{2}\times (t_{0},t_{0}+\tau_{2}L_{2}^{q+1-p} \rho^p).
\end{equation*}
Repeating this argument finitely, there exist positive numbers $\eta_{N}$ and $\tau_{N}$ depending on $n$ and $p$ such that, letting $L_{N}= \inf_{4 B_{N}} u (t_{0})$,
\begin{equation*}
u \geq \eta_{N} L_{N}\quad \textrm{a.e.}\,\,\textrm{in}\,\,B_{j}\times (t_{0},t_{0}+\tau_{N} L_{N}^{q+1-p} \rho^p).
\end{equation*}
for all $j=1,\ldots, N$. 

Finally, we define the subdomain
of $\Omega$ as $\Omega^{\prime\prime}:=\{x \in \Omega :
\mathrm{dist}(x, \Omega^\prime) < 4 \rho \}$,
and put $L :=\inf_{\Omega^{\prime\prime}} u (t_{0}) > 0$.
Then $L \leq L_{i}$ for  $i=1, \ldots, N$. Thus, putting
$\eta_{0}:= \min \limits_{i=1, \ldots, N}\eta_i L$
and
$\tau_{0}:=\min \limits_{i =1, \ldots, N}\tau_{i} L^{q+1-p}\rho^p$, we complete the proof.

\end{proof}



\subsection{Proof of Proposition~\ref{Interior positivity by the volume constraint}}\label{another proof}

We will prove Proposition~\ref{Interior positivity by the volume constraint} here.

\begin{proof}[\normalfont \textbf{Proof of Proposition~\ref{Interior positivity by the volume constraint}}]
Following \cite[Proposition 5.4]{Kuusi-Misawa-Nakamura} with a minor change, we give the proof. 

Note that a nonnegative weak solution $u$ of~\eqref{pS} is a weak supersolution to~\eqref{pST} with $c=0$. 
By the volume constraint together with the boundedness, 
%
letting $M:=e^{\frac{1}{q}\int_0^T\|\nabla u(t)\|_p^p\,dt}\|u_0\|_{\infty}$, we have, for any $t \in [0,T]$
\begin{align}
1=\int_\Omega u^{q+1}(t)\,dx&=\int_{\Omega' \cap \{u(t) \geq L\} }u^{q+1}(t)\,dx+\int_{\Omega' \cap \{u(t) < L\} } u^{q+1}(t)\,dx+\int_{\Omega\,\backslash \Omega'} u^{q+1}(t)\,dx \notag \\
&\leq M^{q+1}\big|\Omega' \cap \{u(t) \geq L\} \big|+L^{q+1}|\Omega'|+M^{q+1}|\Omega\,\backslash \Omega'|\,;\, \notag
\end{align}
i.e.,
\[
\frac{1-L^{q+1}|\Omega'|-M^{q+1}|\Omega\,\backslash \Omega'|}{M^{q+1}} \leq \big|\Omega' \cap \{u(t) \geq L\} \big|.
\]
%
Under the choice of $\Omega^\prime$ and $L$ in Proposition~\ref{Interior positivity by the volume constraint}, we find that, for any $t \in [0,T]$,
\begin{equation}\label{positivity with volume constraint eq1}
\alpha|\Omega'| \leq \big|\Omega' \cap \{u(t) \geq L\} \big|,
\end{equation}
where $\displaystyle \alpha:=\frac{1}{2 M^{q+1}|\Omega'|}$. 
Let $\rho>0$ be arbitrarily taken and fixed, satisfying $\rho \leq \mathrm{dist}(\Omega^\prime, \partial \Omega)/16$. We choose subdomain $\Omega^{\prime\prime}$ as $\Omega^{\prime\prime} =\left\{x \in \Omega : \mathrm{dist}(x, \Omega^\prime) < 4\rho\right\}$.
By Theorem~\ref{ex. thm prime}, there are positive integer $N=N(\Omega^\prime, \rho)$, positive real number families $\delta_{0}, \delta_{N},\,\eta_{N}, \eta_{N+1}, \sigma_{N} \in (0,1),\,J_{N},\,I_{N}\in \mathbb{N}$ depending on $\alpha, N, n, p$ and independent of $L$ and a time $t_N>t_0$ such that
\begin{equation*}
u \geq \eta_{N+1} L \quad \textrm{a.e.}\,\,\textrm{in}\quad \Omega^{\prime}\times \mathcal{I}_{N}(t),
\end{equation*}
where
\begin{equation*} 
\mathcal{I}_{N}(t):=\left(t_{N}+ (1-\sigma_N) \delta_N (\eta_N L)^{q+1-p} \rho^p,\,t_{N}+\delta_N (\eta_N L)^{q+1-p} \rho^p \right), 
\end{equation*}
with $\sigma_N=e^{-(\tau_N + 2e^{\tau_N})}$, $e^{\tau_N} = C(n,p)2^{I_N+J_N}$ and $t_{N}$ given by
\begin{equation*}
t_N=t+(\delta_0-\delta_N\eta_N^{q+1-p})L^{q+1-p}\rho^p.
\end{equation*}
Notice that the terminal time of $\mathcal{I}_{N}(t)$ is $t+\delta_{0}L^{q+1-p}\rho^p$.
Meanwhile, it follows from $u_0>0$ in $\Omega$ and Proposition~\ref{ex. cor prime} with $t_{0}=0$ that, there exist positive number $\eta$ and $\tau$ depending only on $N, n$ and $p$ such that
\[
u \geq \eta L \quad \textrm{a.e.}\,\,\textrm{in}\quad \Omega^\prime \times (0,\tau L^{q+1-p} \rho^p).
\]
Furthermore, if $t$ is very close to $T$, then we can choose $\delta_0 > 0$ so small that
$t+\delta_0 L^{q+1-p}\rho^p = T$. Since $t \in [0, T]$ is arbitrary, choosing $\bar{\eta}:=\min\{\eta_{N+1}, \eta\}$, we have that
\[
u(x,t)\geq \bar{\eta}L  \quad \textrm{a.e.}\,\,\textrm{in}\quad \Omega^{\prime } \times [0,T],
\]
which is our assertion of Proposition~\ref{Interior positivity by the volume constraint}.
\end{proof}

\section{Convergence result}\label{Sec. Convergence result}

We will recall the fundamental convergence result, used in Appendix~\ref{Sec. rigorous argument}. A weak convergent sequence in $L^{r},\,1<r<\infty$, satisfying the norm convergence is strong convergent in $L^r$, that was originally proved by Clarkson (\cite{Clarkson}) and Hanner (\cite{Hanner}).
\begin{lem}\label{strong conv. lemma}
Let $r \in (1,\infty)$. Let $\{f_{j}\}_{j=1}^{\infty}$ be a sequence in $L^{r}(\Omega)$. Assume that $f_{j} \to f$ weakly in $L^{r}(\Omega)$ and $\|f_{j}\|_{r} \to \|f\|_{r}$. Then we have $f_{j} \to f$ strongly in $L^{r}(\Omega)$.
 \end{lem}
\begin{proof}
It is a consequence of the uniform convexity of $L^p$-space. We omit the details of proof.
\end{proof}
\section{Proof of Proposition~\ref{Nonlinear intrinsic scaling}}\label{Sec. rigorous argument}

In this appendix, we shall prove Proposition~\ref{Nonlinear intrinsic scaling}. From now on, we will show that the function $u$ defined by~\eqref{def. of u} satisfies the conditions (D1)--(D4) in Definition~\ref{def of weak sol.}. Firstly, we introduce the mollifier, which is used later. 
For a function $f \in L^{1}(\mathbb{R}^{n+1})$, we denote the \textit{mollifier} of $f$ by
\[
f_{\varepsilon}(z):=(f \ast \rho_{\varepsilon})(z)=\int_{\mathbb{R}^{n+1}}f(w)\rho_{\varepsilon}(z-w)\,dw.
\]
Here, $\varepsilon>0$ and $z=(x,t), w=(y,s) \in \mathbb{R}^{n+1}$ are space-time points, and let $\displaystyle \rho_{\varepsilon}(z):=\frac{1}{\varepsilon^{n+1}}\rho\left(\frac{z}{\varepsilon} \right)$, where $\rho(z)$ is a smooth symmetric function in the following sense:
\[
\rho(x,t)=\rho(|x|,|t|) \quad \textrm{for}\quad z=(x,t) \in \mathbb{R}^{n+1}
\]
and satisfies 
\[
\int_{\mathbb{R}^{n+1}}\rho(z)\,dz=1,\quad \mathrm{supp} \,\rho \subset \{(x,t): |x|<1, |t|<1\}, \quad \rho \geq 0.
\]
In what follows, let $v$ be a weak solution of~\eqref{bpS} with the initial data $v_0=u_0$, obtained from Theorem~\ref{our theorem}. 
We extend $v$ as $v=u_0$ for $t \leq 0$, and $v=0$ outside $\Omega$ and then the extended function is also written by the same notation. Note that this extension is Lipschitz extension of $v$. Let us denote the mollification of $v$ by $v_\varepsilon$.

\begin{lem}\label{crucial uniform convergence}
We have the following uniform convergences:
As $\varepsilon \searrow 0$,

\begin{equation}\label{convergence of v1}
v_{\varepsilon} \to v \quad \textrm{locally uniformly in} \quad C([0,\infty)\,;\,L^{q+1}(\Omega))
\end{equation}
and
\begin{equation}\label{convergence of v2}
(v^q)_{\varepsilon} \to v^q \quad \textrm{locally uniformly in} \quad C([0,\infty)\,;\,L^{\frac{q+1}{q}}(\Omega)).
\end{equation}
\end{lem}

\begin{proof}
From the energy equality~\eqref{energy eq1 of v} and the continuity of $x^{\frac{1}{q+1}}$ for $x \ge 0$,  we find that $\|v(s)\|_{q+1}$ is locally continuous in $s \in [0, \infty)$, in fact, locally absolutely continuous on $[0,\infty)$. This together with Lemma~\ref{strong conv. lemma} and~\eqref{energy eq1 of v}
implies that $v \in C([0,\infty)\,;\,L^{q+1}(\Omega))$. Via the fundamental property of mollifier,~\eqref{convergence of v1} immediately follows. By the same argument as above,~\eqref{convergence of v2} readily follows.
\end{proof}
\medskip

Let $S^\ast$ be the extinction time of this solution $v$ of~\eqref{bpS+}. Define $\Lambda(\tau)$ by a $C^{1}$-solution to the following ODE
\begin{equation}\label{def. of Lambda'}
\begin{cases}
\displaystyle \frac{d}{d\tau}\Lambda(\tau)=\left( \displaystyle \int_{\Omega} v^{q+1}(x,\Lambda(\tau)) \, dx \right)^{\frac{p}{n}}\\[5mm]
\Lambda(0)=0.
\end{cases}
\end{equation}
Let $g(t)$ be a $C^1$-solution of the ODE on $[0,\infty)$
\begin{equation}\label{def. of g'}
\begin{cases}
g^\prime (t)=e^{\Lambda(g(t))}\\[2mm]
g (0)=0.
\end{cases}
\end{equation}
For unique $\Lambda \in C^1(0,\infty)$ solving~\eqref{def. of Lambda'} and, subsequently, $g \in C^1(0, \infty)$ solving~\eqref{def. of g'}, let 
\begin{equation*}
s(t)=S^\ast \left(1-e^{-\Lambda(g(t))}\right)
\end{equation*}
and set
\begin{equation}\label{def. of u'}
u(x,t) := \frac{v\left(x,s(t)\right)}{\gamma(t)}, \quad \gamma(t) := \left(\,\displaystyle \int_{\Omega} v^{q+1}\left(x,s(t)\right) \, dx \right)^{\frac{1}{q+1}}.
\end{equation}
The ODE~\eqref{def. of Lambda'} is actually solvable, since the integral of the right hand side of the ODE~\eqref{def. of Lambda} is (locally absolutely) continuous on $t$ in $[0, \infty)$, by the energy equality~\eqref{energy eq1 of v}. We further remark that the following relation holds:
\[
s\nearrow S^\ast \iff t \nearrow \infty.
\]
Let $t_0<\infty$ be any positive number and set $s_0:=S^\ast\left(1-e^{-\Lambda(g(t_0))}\right)$.
We now deduce the positivity of $\|v (s)\|_{q+1}$ for any nonnegative $s\leq s_0$. 
\begin{lem}\label{bounded from below lem}
There exists a positive number $c_{0}$ such that, for every nonnegative $s \leq s_{0}$,
\begin{equation}\label{c0}
\|v(s)\|_{q+1}\geq c_{0}.
\end{equation}  

\end{lem}

\begin{proof}
Since $\|v(s)\|_{q+1}$ is continuous and positive on $[0,s_{0}]$ there exists a positive number $c_{0}$  satisfying $c_{0}:=\min \limits_{0 \leq s \leq s_{0}}\|v(s)\|_{q+1}>0$, yielding~\eqref{c0}. The proof is done.
\end{proof}

\medskip

We have to verify the regularity of the scaled solution $u$, defined by~\eqref{def. of Lambda} and~\eqref{def. of u'}.
\begin{lem}[The regularity of a composite function and its chain rule]\label{composite, chain rule}
 Let $u$ be defined by~\eqref{def. of u'}. Then there holds that 
 \begin{equation}
 u \in C ([0, \infty); L^{q +1} (\Omega))
 ,\quad \nabla u \in L^\infty_{\mathrm{loc}}([0, \infty); L^p (\Omega)).
 \end{equation}
Furthermore, the function $\gamma(t)$ is Lipschitz on $[0, t_0]$ and the weak derivative on time $\partial_t v^q (x, S^\ast\left(1-e^{-\Lambda(g(t))}\right))$ is in $L^2 (\Omega_\infty)$.
In addition, there exists $\partial_t u^q \in L^2 (\Omega_{t_0})$ for any $t_0<\infty$ such that 
\begin{equation}\label{weak sense eq.}
\partial_t u^q=\partial_t v^q \cdot \gamma^{-q}-qv^q\gamma^{-q-1}\gamma^\prime(t)
\end{equation}
in a weak sense.
\end{lem}
\begin{proof}
The first part of this proof deals with $u \in C ([0,\infty); L^{q+1} (\Omega))$. 
From~\eqref{energy eq1 of v} $v\left(x,S^\ast\left(1-e^{-\Lambda(g(t))}\right)\right) \in C([0,\infty); L^{q+1}(\Omega))$ and $\gamma(t) \in C([0,\infty))$  and thus, by the very definition~\eqref{def. of u'}, $u(x,t)\in C([0,\infty); L^{q+1}(\Omega))$.

Note that 
\[
s(t)=S^\ast\left(1-e^{-\Lambda(g(t))}\right) \quad \iff \quad t=(\Lambda\circ g)^{-1}\left(\log\left(\dfrac{S^\ast}{S^\ast-s}\right)\right)
\]
and, by~\eqref{def. of Lambda'} and~\eqref{def. of g'},
\begin{align*}
s_t=\dfrac{ds}{dt}&=S^\ast e^{-\Lambda(g(t))}\dfrac{d}{dt} \Lambda(g(t))\\[2mm]
&=S^\ast e^{-\Lambda(g(t))}\Lambda^\prime(g(t))g^\prime(t)=S^\ast\Lambda^\prime(g(t))=\|v(s)\|_{q+1}^{(q+1)\frac{p}{n}}.
\end{align*}
Thus, by the changing of variable $s=s(t)=S^\ast\left(1-e^{-\Lambda(g(t))})\right)$ and integration by parts, one has for any $\varphi \in C^\infty_0(\Omega_\infty)$, 
\begin{align*}
&\int_0^\infty\!\!\!\int_\Omega u(x,t)\nabla \varphi(x,t)\,dxdt \notag\\[2mm]
&=\int_0^{S^\ast}\frac{v(x,s)}{\|v(s)\|_{q+1}}\nabla \varphi\left(x, (\Lambda\circ g)^{-1}\left(\log\left(\tfrac{S^\ast}{S^\ast-s}\right)\right)\right)\,dx\,\|v(s)\|_{q+1}^{-(q+1)\frac{p}{n}}\,ds\notag\\[2mm]
&=\int_0^{S^\ast}\|v(s)\|_{q+1}^{-1-(q+1)\frac{p}{n}}\left(-\int_\Omega \nabla v(x,s)\varphi\left(x,(\Lambda\circ g)^{-1}\left(\log\left(\tfrac{S^\ast}{S^\ast-s}\right)\right)\right)\,dx\right)\,ds \notag\\[2mm]
&=\int_0^{S^\ast}\!\!\!\int_\Omega \frac{\nabla v(x,s)\varphi\left(x,(\Lambda\circ g)^{-1}\left(\log\left(\tfrac{S^\ast}{S^\ast-s}\right)\right)\right)}{\|v(s)\|_{q+1}}\,dx\,\|v(s)\|_{q+1}^{-(q+1)\frac{p}{n}}ds,
\end{align*}
since $\nabla v \in L^\infty((0,\infty); L^p(\Omega))$ by~\eqref{energy ineq2 of v} and $\supp \varphi \left(x, (\Lambda\circ g)^{-1}\left(\log\left(\frac{S^\ast}{S^\ast-s}\right)\right)\right) \Subset \Omega_{S^\ast}$.
Again, the changing of variable $t=(\Lambda\circ g)^{-1}\left(\log\left(\dfrac{S^\ast}{S^\ast-s}\right)\right)$ yields
\begin{equation*}
\int_0^\infty\!\!\!\int_\Omega u(x,t)\nabla \varphi(x,t)\,dxdt=-\int_0^\infty\!\!\!\int_\Omega  \frac{\nabla v\left(x,s(t)\right)}{\gamma(t)}\varphi(x,t)\,dxdt
\end{equation*}
and thus, there exists a weak derivative $\nabla u(x,t)$ such that, for any nonnegative $t_0< \infty$,
\begin{equation}\label{nabla u weak sense}
\nabla u(x,t) =\frac{\nabla v(x,s(t))}{\gamma(t)} \in L^\infty([0,\infty); L^p(\Omega)).
\end{equation}
\smallskip

In order to prove~\eqref{weak sense eq.} being valid in the weak sense, we verify that $\gamma(t)$ and $v^q\left(x,s(t)\right)$ are weak differentiable in $t$ and their weak derivatives are integrable on $(0, t_0)$ and $\Omega_\infty$, respectively.

Firstly, we show $\gamma(t)$ is weak differentiable in $(0, \infty)$.
Now, set $f(s):=\|v(s)\|_{q+1}^{q+1}$. From~\eqref{energy eq1 of v} $f(s)$ is a locally absolutely continuous function on $s\in [0, \infty)$ and for a.e. $s \in [0,\infty)$, 
\begin{equation}\label{composite, chain rule eq.0}
\frac{d}{ds}f(s)=\frac{d}{ds}\|v(s)\|_{q+1}^{q+1}=-\frac{q+1}{q}\|\nabla v(s)\|_p^p
\end{equation}
and $\frac{d}{ds}f(s)$ is bounded on $(0,\infty)$ by~\eqref{energy ineq2 of v} and thus, $f(s)$ is actually Lipschitz function. Now, $f_h(s)$ and $\frac{d}{ds}f_h(s)$ denote the mollification with respect to time variable $s$ of $f(s)$ and $\frac{d}{ds}f(s)$, respectively.
According to the fundamental property of mollifier, we have, as $h \searrow 0$,
\begin{equation}\label{f convergence}
f_h(s) \to f(s)\quad \textrm{locally uniformly}\quad \textrm{on}\,\,[0,\infty)
\end{equation}
and, for all $r \geq 1$,
\begin{equation}\label{f' convergence}
\frac{d}{ds}f_h(s) \to \frac{d}{ds}f(s) \quad \textrm{strongly}\quad \textrm{in}\,\,L^r(0,\infty).
\end{equation}
By integration by parts we see that, for every $\phi(t) \in C_0^\infty((0,\infty))$, 
\begin{equation}\label{composite, chain rule eq.1}
\int_0^\infty\phi(t)\big[f_h\left(s(t)\right)\big]^\prime\,dt=-\int_0^\infty\phi^\prime(t)f_h\left(s(t)\right)\,dt,
\end{equation}
where $^\prime:=\frac{d}{dt}$. By changing of the variable $s=s(t)=S^\ast\left(1-e^{-\Lambda(g(t))}\right)$, the integration in the left hand side of~\eqref{composite, chain rule eq.1} is computed as
\begin{align*}
\int_0^{\infty}\phi(t)\big[f_h(s(t))\big]^\prime\,dt&=\int_0^\infty\phi(t)\frac{d}{ds}f_h\left(s(t)\right)\underbrace{s_t\,dt}_{=ds} \notag\\
&=\int_0^{S^\ast}\phi\left((\Lambda\circ g)^{-1}\left(\log\left(\tfrac{S^\ast}{S^\ast-s}\right)\right)\right)\frac{d}{ds}f_h(s)\,ds, 
\end{align*}
which converges to 
\begin{equation}\label{composite, chain rule eq.2}
\int_0^{S^\ast}\phi\left((\Lambda\circ g)^{-1}\left(\log\left(\tfrac{S^\ast}{S^\ast-s}\right)\right)\right)\frac{d}{ds}f(s)\,ds =\int_0^\infty\phi(t)\frac{d}{ds}f\left(s(t)\right)s_t\,dt
\end{equation}
as $h \searrow 0$, by~\eqref{f' convergence}.
On the other hand, the integration in the right hand side of~\eqref{composite, chain rule eq.1} converges to
\begin{equation}\label{composite, chain rule eq.3}
-\int_0^\infty\phi^\prime(t)f_h\left(s(t)\right)\,dt\,\,\to\,\, -\int_0^\infty\phi^\prime(t)f\left(s(t)\right)\,dt
\end{equation}
as $h \searrow 0$, by~\eqref{f convergence}. 
From~\eqref{composite, chain rule eq.1},~\eqref{composite, chain rule eq.2} and~\eqref{composite, chain rule eq.3} we obtain
\begin{equation*}
\int_0^\infty\phi(t)\,s_t\frac{d}{ds}f\left(s(t)\right)\,dt=-\int_0^\infty\phi^\prime(t)f\left(s(t)\right)\,dt
\end{equation*}
and thus, there exists a weak derivative $\frac{d}{dt}f(s(t))$ in $(0,\infty)$ such that 
\begin{equation}\label{composite, chain rule eq.4-1}
\frac{d}{dt}f\left(s(t)\right)=s_t\frac{d}{ds}f\left(s(t)\right)
\end{equation}
and therefore, by~\eqref{energy ineq1 of v} and~\eqref{def. of Lambda}, $f\left(s(t)\right)$ is Lipschitz on $[0, \infty)$.
From~\eqref{composite, chain rule eq.4-1}, $\gamma(t)$ is weak differentiable in $(0, t_0)$ and 
\begin{align}\label{composite, chain rule eq.4-2}
\gamma^\prime(t)&=\frac{d}{dt}\left\|v\left(s(t)\right)\right\|_{q+1} \notag\\[2mm]
&=\frac{1}{q+1}\left\|v\left(s(t)\right)\right\|_{q+1}^{-q}\frac{d}{dt}f\left(s(t)\right) \notag\\[2mm]
&=\frac{1}{q+1}\left\|v\left(s(t)\right)\right\|_{q+1}^{-q}\,s_t\frac{d}{ds}f\left(s(t)\right),
\end{align}
because $\gamma^{\frac{1}{q+1}}$ is Lipschitz for $\gamma \geq c_0$.
From~\eqref{energy ineq1 of v},~\eqref{energy ineq2 of v},~\eqref{def. of Lambda},~\eqref{c0},~\eqref{composite, chain rule eq.0} and~\eqref{composite, chain rule eq.4-2}, it follows that, for every $t \in (0, t_0)$,
\begin{align}\label{composite, chain rule eq.4-3}
|\gamma^\prime(t)| 
 \leq \frac{c_0^{-q}}{q}\|u_0\|_{q+1}^{(q+1)\frac{p}{n}}\|\nabla u_0\|_p^p
\end{align}
and thus, $\gamma (t)$ is surely Lipschitz on $[0, t_0]$.
\medskip

Next, we will verify that the weak derivative on time of $v^q\left(x,S^\ast\left(1-e^{-\Lambda(g(t))}\right)\right)$ is in $L^2 (\Omega_\infty)$. From Definition~\ref{def of v} it follows that $\partial_sv^q(x,s) \in L^2(\Omega_\infty)$ and thus, 
\begin{equation}\label{der. vq convergence}
\partial_s(v^q)_\varepsilon \to \partial_s v^q \quad \textrm{strongly}\quad \textrm{in}\,\,L^2(\Omega_\infty).
\end{equation}
Again, by integration by parts we see that, for every $\varphi\in C_0^\infty(\Omega_\infty)$, 
\begin{equation}\label{composite, chain rule eq.6}
\int_0^\infty\!\!\!\int_\Omega\varphi \, \partial_t(v^q)_\varepsilon\,dxdt=-\int_0^\infty\!\!\!\int_\Omega\partial_t\varphi \,(v^q)_\varepsilon\,dxdt.
\end{equation}
By $s=s(t)=S^\ast\left(1-e^{-\Lambda(g(t))}\right)$ and~\eqref{der. vq convergence}, the integration in the left hand side of~\eqref{composite, chain rule eq.6} is computed as
\begin{align*}
\int_0^\infty\!\!\!\int_\Omega \varphi \,\partial_t (v^q)_\varepsilon\,dxdt&=\int_0^\infty\!\!\!\int_\Omega \varphi(x,t)\partial_s(v^q)_\varepsilon\left(x, f\left(s(t)\right)\right) s_t\,dxdt \notag\\
&=\int_0^{S^\ast}\!\!\!\int_\Omega\varphi\left(x,(\Lambda\circ g)^{-1}\left(\log\left(\tfrac{S^\ast}{S^\ast-s}\right)\right)\right)\,\partial_s(v^q)_\varepsilon (x,s)\,dxds,
\end{align*}
which converges to
\begin{align}\label{composite, chain rule eq.7}
&\int_0^{S^\ast}\!\!\!\int_\Omega \varphi\left(x,(\Lambda\circ g)^{-1}\left(\log\left(\tfrac{S^\ast}{S^\ast-s}\right)\right)\right)\,\partial_s(v^q)(x,s)\,dxds \notag\\[2mm]
&\quad \quad \quad =\int_0^\infty\!\!\!\int_\Omega\varphi(x,t)\,\partial_sv^q\left(x, s(t)\right) s_t\,dxdt
\end{align}
as $\varepsilon \searrow 0$.
The integration in the right hand side of~\eqref{composite, chain rule eq.6} converges as 
\begin{equation}\label{composite, chain rule eq.8}
-\int_0^{\infty}\!\!\!\int_\Omega\partial_t\varphi \, (v^q)_\varepsilon\,dxdt\,\,\to\,\, -\int_0^\infty\!\!\!\int_\Omega\partial_t \varphi \,v^q\,dxdt
\end{equation}
as $\varepsilon \searrow 0$, by~\eqref{convergence of v2}. 
By~\eqref{composite, chain rule eq.6},~\eqref{composite, chain rule eq.7} and~\eqref{composite, chain rule eq.8} we have
\begin{equation*}
\int_0^\infty\!\!\!\int_\Omega\varphi(x,t)\partial_sv^q\left(x, s(t)\right) s_t\,dxdt=-\int_0^\infty\!\!\!\int_\Omega\partial_t \varphi \,v^q\,dxdt
\end{equation*}
and thus, there exists a weak derivative $\partial_tv^q \left(x, s(t)\right)$ in $\Omega \times (0, \infty)$ such that 
\begin{equation}\label{composite, chain rule eq.9-1}
\partial_tv^q=\partial_sv^q\left(x, s(t)\right) s_t.
\end{equation}
By~\eqref{energy ineq1 of v},~\eqref{def. of Lambda} and~\eqref{composite, chain rule eq.9-1}
\begin{equation}\label{composite, chain rule eq.9-2}
|\partial_t v^q| \leq \|u_0\|_{q+1}^{(q+1)\frac{p}{n}} \Big|\partial_sv^q\left(x, S^\ast\left(1-e^{-\Lambda(g(t))}\right)\right)\Big|
\end{equation}
and thus, the weak derivative on time of $v^q \left(x, s(t)\right)$ is in $L^2 (\Omega_\infty)$.
\medskip

Note that $\gamma \mapsto \gamma^{-q}$ is locally Lipschitz on $\{\gamma \geq c_0\}$ and $\gamma(t)^{-q}$ is the composite function of $\gamma^{-q}$ with $\gamma(t)$. This together with~\eqref{def. of u},~\eqref{composite, chain rule eq.4-2},~\eqref{composite, chain rule eq.4-3},~\eqref{composite, chain rule eq.9-1} and~\eqref{composite, chain rule eq.9-2} yields that there exists $\partial_t u^q \in L^2 (\Omega_{t_0})$ for any $t_0<\infty$ such that 
\begin{equation*}
\partial_t u^q=\partial_t v^q \cdot \gamma^{-q}-qv^q\gamma^{-q-1}\gamma^\prime(t)
\end{equation*}
in a weak sense, which is the desired result~\eqref{weak sense eq.}.

\end{proof}
\begin{rmk}\normalfont
Let $t_0<\infty$ be any positive number. By the very definition~\eqref{def. of u} of $u$, one has $\|u(t)\|_{q+1}=1$ for $t\in [0,t_0]$ and $0 \leq u\leq c_0^{-1}\|u_0\|_\infty$ for every $(x,t) \in \Omega \times [0,t_0]$ via~\eqref{c0},~\eqref{MP} and the nonnegativity of $v$. Applying the same argument as the proof of~\cite[Proposition 5.2]{Kuusi-Misawa-Nakamura}, we readily get $\lambda(t)=\|\nabla u(t)\|_p^p$ for every $t \in [0,t_0]$ and thus, by~\eqref{c0},~\eqref{nabla u weak sense} and~\eqref{energy ineq2 of v}, $\lambda(t) \leq c_0^{-p}\|\nabla u_0\|_p^p <\infty$. Therefore the following equation holds true in a weak sense,
and almost everywhere in $\Omega_{t_0}$:
\begin{equation}\label{pSt0}
\partial_t u^q-\Delta_pu=\lambda(t)u^q \quad \textrm{in}\,\,\,\Omega_{t_0}
\end{equation}
and thus, it plainly holds that $\Delta_p u \in L^2(\Omega_{t_0})$.
\end{rmk}
\medskip

Here we obtain the interior positivity with $c_0$ in~\eqref{c0}  by the volume constraint.

\begin{prop}[Interior positivity with $c_0$ in~\eqref{c0} by the volume constraint]\label{Interior positivity c0 by the volume constraint}
Let the initial data $u_0 \in W^{1,p}_{0}(\Omega)$ be positive, bounded in $\Omega$ and satisfy $\|u_{0}\|_{q+1}=1$. Let $u$ be a nonnegative weak solution of~\eqref{pSt0}  in $\Omega_{t_0}$ with any positive $t_0<\infty$. Put $\widetilde{M}:=e^{c_0^{-p}\|\nabla u_{0}\|_{p}^{p}t_0/q} \|u_{0}\|_{\infty}$ and let $\Omega^{\prime}$ be a subdomain compactly contained in $\Omega$ satisfying $|\Omega \setminus \Omega^{\prime}| \leq \frac{1}{4\widetilde{M}^{q+1}}$. Then there exists a positive constant $\widetilde{\eta}$ such that %
\[
u(x,t)\geq \widetilde{\eta} L \quad \textrm{in}\quad \Omega^{\prime} \times [0,t_0].
\]
Here $0<L\leq \min \left\{ \left(\frac{1}{4|\Omega^{\prime}|}\right)^{\frac{1}{q+1}}, \inf \limits_{\Omega^{\prime\prime}} u_{0} \right\}$, 
  where $\Omega^{\prime\prime}$ is compactly contained in $\Omega$ and compactly containing $\Omega^\prime$, and  the positive constant $\widetilde{\eta}$ depends only on $p, n, \Omega^{\prime}, M$ and $N$, where $N$ is the number of chain balls of $\Omega^{\prime}$. The constant $\widetilde{\eta}$ also depends on the positive constant $c_0$.
\end{prop}
%
%
\begin{proof}
The proof of this proposition is done by the same argument as Proposition~\ref{Interior positivity by the volume constraint}.
\end{proof}
%

\end{appendices}





\begin{thebibliography}{99}

{\small 
\bibitem{Alt-Luckhaus}
H.W. Alt and S. Luckhaus, Quasilinear Elliptic-Parabolic Differential Equations,   \textit{Math. Z.} \textbf{183}\,(1983), 311--341.

\bibitem{Aubin1}
T. Aubin, Equations diff\'{e}rentielles non lin\'{e}aires et probl\`{e}me de Yamabe concernant la courbure scalaine,  \textit{J. Math. Pures Appl.} \textbf{55} (1976), 269--296.

\bibitem{Aubin2}
T. Aubin, \textit{Some nonlinear problems in Riemannian geometry}, Springer Monographs in Mathematics, 1997.

\bibitem{AKN}
B. Avelin, T. Kuusi and K. Nystr\"om: Boundary behavior of solutions to the parabolic p-Laplace equation. \textit{Anal. PDE} \textbf{12}(1)  (2019), 1--42.


 \bibitem{Barrett-Liu}
J.W. Barrett and W.B. Liu, Finite element of approximation of the parabolic $p$-Laplacian,  \textit{SlAM J. Numer. Anal}, Vol. \textbf{34}(2)\,(1994), 413--428.


\bibitem{Clarkson}
J. A. Clarkson, Uniformly convex spaces, {\it Trans. Amer. Math. Soc.} \textbf{40}   (1936), no. 3, 396--414.


\bibitem{DiBenedetto1}
E. DiBenedetto, \textit{Degenerate parabolic equations}, Universitext, Springer-Verlag, New York, 1993.

\bibitem{DiBenedetto2}
E. DiBenedetto,\,U. Gianazza and V. Vespri, \textit{Harnack's inequality for degenerate and singular parabolic equations}, Springer Monographs in Mathematics, 2012.

\bibitem{Evans}
L.C. Evans, \textit{Partial Differential Equations}, American Mathematical Society, Providence, RI, 1998.

\bibitem{Gianazza-Vespri}
U. Gianazza and V. Vespri, Parabolic De Giorgi classes of order $p$ and the Harnack inequality, \textit{Calc. Var. Partial Differential Equations}, \textbf{26}(3) (2006), 379--399.


\bibitem{Giaquinta}
M. Giaquinta, \textit{Introduction to regularity theory for nonlinear elliptic systems}, Birkh"{a}user Verlag (1993).

\bibitem{Hamilton89}
R.S. Hamilton, \textit{Lectures on geometric flows}, (1989)\,(unpublished)


\bibitem{Hanner}
O. Hanner, On the uniform convexity of $L^{p}$ and $\ell^{p}$, {\it Ark. Mat.} \textbf{3} (1956), 239--244. 

\bibitem{Ivanov1} A. V. Ivanov, Uniform H\"{o}lder estimates for generalized solutions of quasilinear parabolic equations that admit double degeneration, \textit{Algebra i Analiz}, 3(2):139--179, 1991, Translation in  \textit{St. Petersburg Math. J.} \textbf{3} (1992), no. 2, 363--403.

\bibitem{Ivanov2} A. V. Ivanov. H\"{o}lder estimates for a natural class of equations of fast diffusion type H\"{o}lder estimates for equations of fast diffusion type. \textit{Zap. Nauchn. Sem. S.-Peterburg. Otdel. Mat. Inst. Steklov. (POMI)}, 229 (11):29--322, 1995,  Translation in \textit{J. Math. Sci.} (New York) \textbf{89} (1998), no. 6, 1607--1630.


\bibitem{Karim-Misawa}
C. Karim and M. Misawa, Gradient H\"older regularity for nonlinear parabolic systems of p-Laplacian type, \textit{Differential Integral Equations} \textbf{29} (2016), no. 3-4, 201--228.

\bibitem{KMN}
T. Kuusi, G. Mingione and K. Nystr\"om.: A boundary Harnack inequality for singular equations of p-parabolic type, \textit{Proc. Amer. Math. Soc.} 142(8) (2014), 2705-2719.

\bibitem{Kuusi-Siljander-Urbano}
T. Kuusi, J. Siljander and J.M. Urbano, Local H\"{o}lder continuity for doubly nonlinear parabolic equations, \textit{Indiana Univ., Math. J.} \textbf{61}(1) (2012), 399--430.

\bibitem{Kinnunen-Kuusi}
J. Kinnunen and  T. Kuusi, Local behavior of solutions to doubly nonlinear parabolic equations. \textit{Math. Ann.} \textbf{337}(3) (2007), 705--728.

\bibitem{KuusiPisa}
T. Kuusi,  Harnack estimates for weak supersolutions to nonlinear degenerate parabolic equations, \textit{Ann. Sc. Norm. Super. Pisa Cl. Sci. (V)} 7 (2008), 673-716.

\bibitem{Kuusi-Misawa-Nakamura}
T. Kuusi, M. Misawa and K. Nakamura, Regularity estimates for the $p$-Sobolev flow, \textit{Journal of Geometric Analysis, ''Perspectives of Geometric Analysis in PDEs'' }(2019), 1--48.

\bibitem{Misawa}
M. Misawa, Local H\"{o}lder regularity of gradients for evolutional $p$-Laplacian systems, \textit{Ann. Mat. Pura Appl.} \textbf{181} (2002), 389--405. 

\bibitem{Nakamura-Misawa}
K. Nakamura and M. Misawa, Existence of a weak solution to the $p$-Sobolev flow, \textit{Non. Anal. TMA} \textbf{175C} (2018), 157--172.


\bibitem{Porzio-Vespri}
M. M. Porzio and V. Vespri, H\"{o}lder estimates for local solutions of some doubly nonlinear degenerate parabolic equations, \textit{J. Differential Equations}, \textbf{103} (1) (1993),146--178.


\bibitem{Schwetlick-Struwe}
H. Schwetlick and M. Struwe, Convergence of the Yamabe flow for large energies, \textit{J. Reine Angew. Math.} \textbf{562} (2003), 59--100.


\bibitem{Sciunzi}
B. Sciunzi, Classification of positive $\mathcal{D}^{1,p}(\mathbb{R}^N)$- solution to the critical $p$-Laplace equation,\,\, \textit{Adv. Math.} \textbf{291} (2016), 12--23.

\bibitem{Suzuki-Kamioka}
T. Suzuki and Y. Ueoka, Lecture on partial differential equations--a course in semi-linear elliptic equations,\, \textit{Baif\^{u}kan},\,2005 (Japanese).

\bibitem{Talenti}
G. Talenti, Best constant in Sobolev inequality, \textit{Ann. Mat. Pura Appl.}, (4)\textbf{110} (1976), 353--372.

\bibitem{Trudinger}
N.S. Trudinger, Pointwise estimates and quasilinear parabolic equations,  \textit{Comm. Pure Appl. Math.}, \textbf{21} (1968), 205--226.

\bibitem{Urbano}
J.M. Urbano, {\it The method of intrinsic scaling}, Lecture Notes in Mathematics \textbf{1930},\,Springer-Verlag,\,Berlin,\,2008.

\bibitem{Vazquez1}
J.L. Vazquez, {\it The porous medium equation}. Mathematical theory. Oxford Mathematical Monographs. The Clarendon Press, Oxford University Press, Oxford, 2007.

\bibitem{Vazquez2}
J.L. Vazquez, {\it  Smoothing and decay estimates for nonlinear diffusion equations. Equations of porous medium type.} Oxford Lecture Series in Mathematics and its Applications, 33. Oxford University Press, Oxford, 2006.

\bibitem{Vazquez3}
J.L. Vazquez, Personal communication, 2013.


\bibitem{Vespri1} V. Vespri, On the local behavior of solutions of a certain class of doubly nonlinear parabolic equations, \textit{Manuscripta Math.} \textbf{75} (1992), 65--80. 

\bibitem{Vespri2} V. Vespri, Harnack type inequalities for solutions of certain doubly nonlinear parabolic equations. \textit{J. Math. Anal. Appl.}, \textbf{181}(1) \,(1994), 104--131.

\bibitem{Yamabe}
H. Yamabe, On a deformation of Riemannian structures on compact manifolds, \textit{Osaka Math. J.} \textbf{12} (1960), 21--37.

\bibitem{Ye}
R. Ye, Global existence and convergence of Yamabe flow, \textit{J.Diff. Geom.}, \textbf{39}\,\,(1994), 35--50.



}
\end{thebibliography}
 \end{document}